\documentclass[11pt]{article} 
\usepackage[english]{babel}
\usepackage[utf8]{inputenc}
\usepackage[T1]{fontenc}
\usepackage{lmodern}
\usepackage{amsmath}
\usepackage{amsfonts}
\usepackage{amssymb}
\usepackage[all]{xy}
\usepackage{color}
\usepackage{graphicx}
\usepackage[section]{placeins}
\usepackage{hyperref}
\usepackage[normalem]{ulem}

\usepackage{amsthm}
\newtheorem{theorem}{Theorem}
\newtheorem{corollary}[theorem]{Corollary}
\newtheorem{lemma}[theorem]{Lemma}
\newtheorem{proposition}[theorem]{Proposition}
\newtheorem*{remark}{Remark}

\title{On the regularity of fractional integrals of \\ modular forms.}
\author{Carlos Pastor}
\date{}

\let\OLDthebibliography\thebibliography
\renewcommand\thebibliography[1]{
  \OLDthebibliography{#1}
  \setlength{\parskip}{1pt}
  \setlength{\itemsep}{1pt plus 0.3ex}
}

\DeclareMathOperator{\slgroup}{SL}
\DeclareMathOperator{\glgroup}{GL}

\def\CC{{C\nolinebreak[4]\hspace{-.05em}\raisebox{.4ex}{\tiny\bf ++}}}

\begin{document}

\maketitle

\begin{abstract}
In this paper we study some local and global regularity properties of Fourier series obtained as fractional integrals of modular forms. In particular we characterize the differentiability at rational points, determine their Hölder exponent everywhere (using several definitions) and compute the associated spectrum of singularities. We also show that these functions satisfy an approximate functional equation, and use it to discuss the graphs of ``Riemann's example'' and of fractional integrals of cusp forms for $\Gamma_0(N)$. We include some computer plots.
\end{abstract}

\section{Introduction}\label{sec:introd}

The function
\begin{equation}\label{eq:riemann}
\varphi(x) = \sum_{n \geq 1} \frac{\sin(n^2 \pi x)}{n^2}
\end{equation}
was introduced in \cite{weierstrass} by Weierstrass as an example supposedly given by Riemann of a continuous function which is not differentiable anywhere. It was later verified by Hardy \cite{hardy} that this is indeed the case except perhaps at the rational numbers of the form $\text{odd}/\text{odd}$ or $\text{even}/(4n+1)$. The behavior at the remaining points was not known until 1970, when Gerver proved in \cite{gerver, gerver2} that $\varphi$ is actually differentiable at those rationals of the form $\text{odd}/\text{odd}$ while is not in the other case (these assertions for small denominator are apparent from its graph when plotted with the aid of modern computers; see figure~\ref{fig:riemann}). 

\begin{figure}\label{fig:riemann}
\centering
\includegraphics[width=0.9\textwidth]{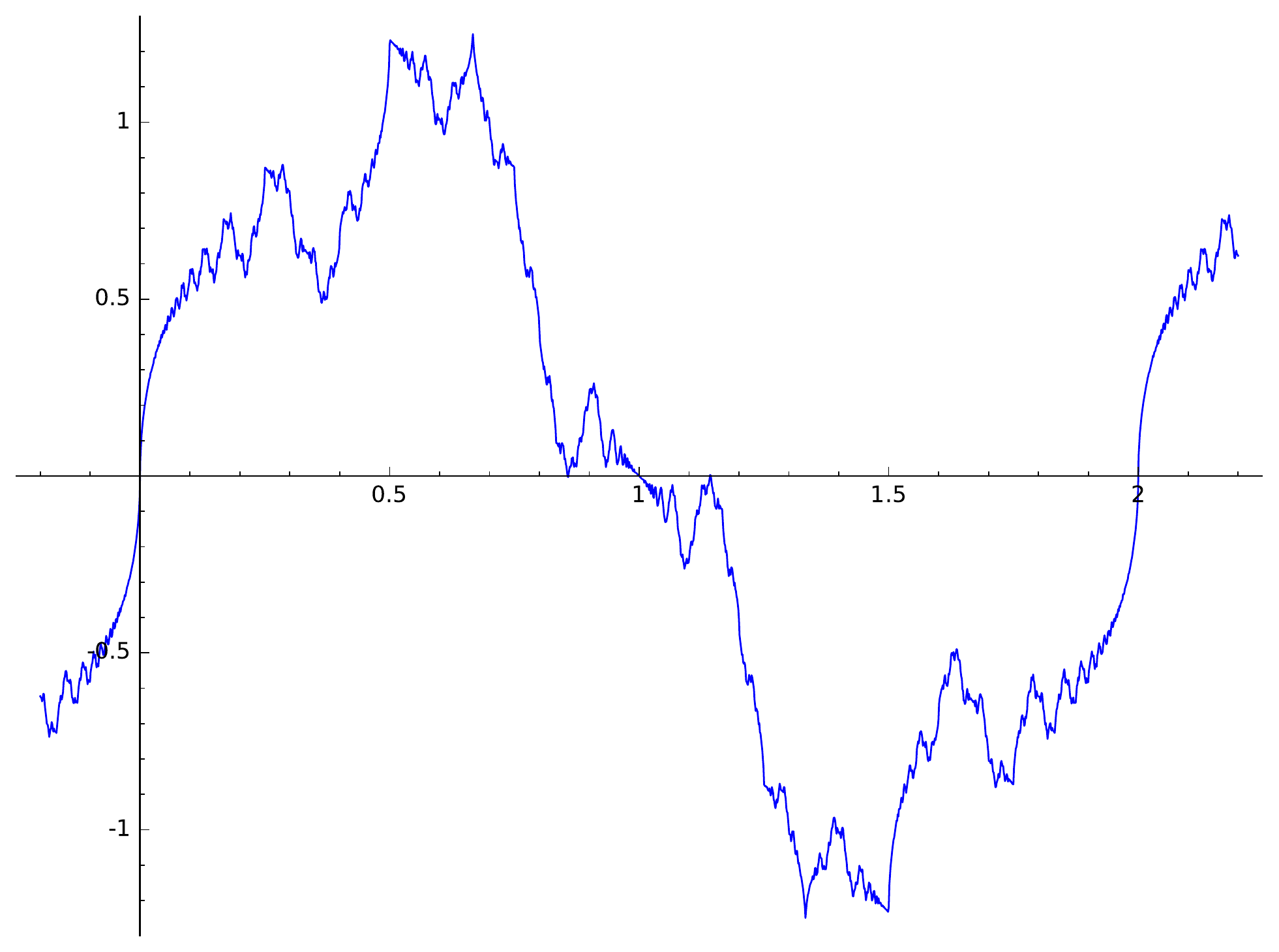}
\caption{Riemann's example.}
\end{figure}

In the light of historical analysis \cite{butzer_stark} it seems probable that Riemann never made such a claim. In spite of this, $\varphi$ has become known in the literature as ``Riemann's example'', and its regularity has been extensively studied by several authors. What lies underneath its apparently chaotic behavior is the action of a certain subgroup of $\slgroup_2(\mathbb{Z})$ on Jacobi's theta function
\[\theta(z) = \sum_{n \in \mathbb{Z}} e^{n^2 \pi i z} \qquad (\Im(z) > 0),\]
which is a modular form of weight $1/2$. The connection can be formally stated as $2\pi^{-1}\varphi'(x) = \Re\theta(x) - 1$. This leads to different fruitful strategies that can be used to study $\varphi$: for example, in \cite{holschneider_tchamitchian} and \cite{jaffard} the derivative in the left hand side is understood as a certain wavelet transform, and general theorems are applied which relate bounds on one side of the transform with regularity on the other; while in \cite{duistermaat} an approximate functional equation for $\varphi$ is deduced integrating the one for $\theta$.

In the same spirit, if we are given any modular form of positive weight for a subgroup of finite index of $\slgroup_2(\mathbb{Z})$ one can perform a formal fractional integration on its Fourier series, obtaining a new series which converges to a continuous function on the whole real line. Several aspects concerning the regularity of the resulting function have been studied for certain limited ranges in \cite{chamizo} and more recently \cite{chamizo_petrykiewicz_ruiz}. The so called pointwise Hölder exponent, for example, which measures how well a function can be locally approximated by polynomials, was determined by Chamizo, Petrykiewicz and Ruiz-Cabello under very restrictive hypothesis in \cite{chamizo_petrykiewicz_ruiz}.

In this paper we intertwine the wavelet transform and approximate functional equation approaches in order to study the pointwise Hölder exponent and related notions under general hypothesis, achieving a more complete understanding of the regularity of these functions. We also determine the differentiability at the rational points, and the spectrum of singularities, which consists of the mapping associating to each of the sets where the function attains a particular Hölder exponent its Hausdorff dimension. To compute the latter a version of the Jarník-Besicovith theorem on Diophantine approximation adapted to the set of rational numbers where the modular form is not cuspidal is needed. For the precise definitions see \S\ref{sec:results}.

As noted by Duistermaat in \cite{duistermaat}, the approximate functional equation also deepens our understanding of the graph of these functions. In particular it shows that around every rational number we should expect oscillations of the form $x^a g(1/x)$, where $g$, which depends on the rational number, is a periodic function also given by a fractional integral of a modular form (these oscillations are clearly visible in figure~\ref{fig:riemann}). Other features of these graphs are also unveiled by the approximate functional equation, such as non-differentiable singularities of the form $x^a$ or $x^a\log{x}$, and self-similarity around quadratic surds and certain rational numbers. To illustrate these concepts we perform a detailed analysis of ``Riemann's example'', recovering several results that appear in the literature.

We also consider the case of cusp forms for the group $\Gamma_0(N)$, that is, the subgroup of $\slgroup_2(\mathbb{Z})$ consisting of those matrices whose bottom-left entry is divisible by $N$. These modular forms are relevant because of the relation with elliptic curves (and other abelian varieties) via the acclaimed modularity theorem. It turns out that the question of determining around which rational numbers the aforementioned self-similarity takes place is linked to understanding when the normalizer of $\Gamma_0(N)$ acts transitively on $\mathbb{Q}$. We characterize the latter and provide sufficient conditions for the former. This is one more example of the fact that some algebro-geometric properties affect the aspect of these fractional integrals. For example, in \cite{chamizo} it is shown that under some circumstances the derivative at $0$ of these functions can be related to the rank of the associated elliptic curve.

The layout of this paper is the following: in \S\ref{sec:results} we formally define all the necessary concepts and state our main results. \S\ref{sec:lemmas} contains some preliminary lemmas, while \S\ref{sec:approx_eq} and \S\ref{sec:wavelet} are devoted to developing the main tools: the approximate functional equation and the wavelet transform. In \S\ref{sec:proofs} we focus on determining several Hölder exponents, while \S\ref{sec:spectrum} deals with the spectrum of singularities. Finally, in \S\ref{sec:riemann} and \S\ref{sec:gamma0} the previous theory is applied to ``Riemann's example'' and cusp forms for $\Gamma_0(N)$, respectively. These last sections also include computer-generated images.

\section{Notation and main results}\label{sec:results}

As it is customary, the notation $f(x) \ll g(x)$, or $f = O(g)$, will be employed to denote that the inequality $|f(x)| \leq C |g(x)|$ is satisfied for some unspecified positive constant $C$, usually as $x$ converges to a certain point.

We introduce the following spaces of complex-valued functions, defined either in all $\mathbb{R}$ or in an open subset of $\mathbb{R}$, to classify them according to their regularity:

\begin{itemize}
\item
For $0 \leq s \leq 1$ we define $\Lambda^s(x_0)$ as the set of all continuous functions which satisfy a $s$-Hölder condition at $x_0$, \emph{i.e},
\[|f(x) - f(x_0)| \ll |x-x_0|^s \qquad (x \rightarrow x_0).\]
We analogously define $\Lambda^s(\Omega)$ for a subset $\Omega \subset \mathbb{R}$ as the set of all continuous functions satisfying a uniform $s$-Hölder condition on $\Omega$.
\item
For any $s \geq 0$ we define $\mathcal{C}^s(x_0)$ as the set of all continuous functions for which there is some polynomial $P$ of degree less than $s$ satisfying
\[|f(x) - P(x-x_0)| \ll |x-x_0|^s \qquad (x \rightarrow x_0).\]
Note that for $0 \leq s \leq 1$ the spaces $\mathcal{C}^s(x_0)$ and $\Lambda^s(x_0)$ coincide.
\item
For any $0 \leq s \leq 1$ and any integer $k \geq 0$ we define $\mathcal{C}^{k, s}(x_0)$ as the set of all continuous functions for which $f^{(k)}$ exists in an interval $I$ containing $x_0$ and verifies $f^{(k)} \in \Lambda^s(x_0)$. Analogously one defines $\mathcal{C}^{k, s}(\Omega)$ for an open set $\Omega \subset \mathbb{R}$ as the set of all continuous functions for which $f^{(k)}$ exists in $\Omega$ and $f^{(k)} \in \Lambda^s(K)$ for every compact subset $K \subset \Omega$.
\end{itemize}

Finally we also define the spaces $\Lambda^s_{\log}$, $\mathcal{C}^s_{\log}$ and $\mathcal{C}^{k, s}_{\log}$ by replacing $|x-x_0|^s$ in the previous definitions with $|x-x_0|^s\log|x-x_0|$.

Motivated by the articles \cite{chamizo_petrykiewicz_ruiz, seuret_vehel}, we choose the following Hölder exponents as measures of the local regularity of a function $f$ at a certain point:
\begin{align*}
\beta(x_0) &:= \sup \{s : f \in \mathcal{C}^s(x_0)\}, \\
\beta^{*}(x_0) &:= \sup \{k+s : f \in \mathcal{C}^{k, s}(x_0)\}, \\
\beta^{**}(x_0) &:= \lim_{I \rightarrow \{x_0\}} \sup \{k+s : f \in \mathcal{C}^{k, s}(I)\}.
\end{align*}
In the last definition the limit is taken as $I$ runs over a sequence of nested open intervals whose intersection is $\{x_0\}$.

The first exponent, $\beta(x_0)$, also called the pointwise Hölder exponent, is the most local in nature and gives precise information about how well the function can be approximated by a polynomial in arbitrarily small neighborhoods of $x_0$, even when no derivative exists near that point (note that $P$ in the definition of $\mathcal{C}^s(x_0)$ need not be the Taylor polynomial).

The second one, $\beta^{*}(x_0)$, also called the restricted local Hölder exponent, is more demanding in the sense that $f$ must be differentiable enough times for it to coincide with $\beta(x_0)$. This is in some sense like imposing that the polynomial is the actual Taylor polynomial of $f$.

Finally, $\beta^{**}(x_0)$, the local Hölder exponent, requires $f$ not only to be differentiable in open neighborhoods, but also its derivatives to satisfy a Hölder condition in them. The importance of this last one resides in the fact that it behaves well under the action of a wide class of pseudo-differential operators (see \cite{seuret_vehel}).

It is not hard to prove that these exponents satisfy the inequalities
\[\beta(x) \geq \beta^{*}(x) \geq \beta^{**}(x),\]
and that there are examples for which both inequalities are strict (it suffices to consider functions of the form $x^a\sin{x^{-b}}$; see \cite{chamizo_petrykiewicz_ruiz}).

Unless otherwise stated from now on $f$ will denote a nonzero modular form of weight $r$ for a subgroup $\Gamma$ of finite index of $\slgroup_2(\mathbb{Z})$ and multiplier system $\{\mu_\gamma\}$ (\emph{cf.} \cite{iwaniec, rankin}). This means that $f$ is analytic in the upper half-plane $\mathbb{H}$, transforms by the law
\begin{equation}\label{eq:modularity}
f(\gamma z) = \mu_\gamma (cz+d)^r f(z) \quad \text{for every } \gamma = \left( \begin{matrix} a & b \\ c & d \end{matrix} \right) \in \Gamma,
\end{equation}
and has at most polynomial growth when $\Im{z} \rightarrow 0^+$. The term $\gamma z$ stands for the Möbius transformation $(az+b)/(cz+d)$, while $\mu_\gamma$ is a unimodular complex number associated to the matrix $\gamma$. We also introduce the classical notation $j_\gamma(z) := cz+d$. All the power and logarithm functions considered in this article correspond to the branch with argument determination $-\pi < \arg{z} \leq \pi$.

Given any matrix $\gamma$ in $\glgroup_2^{\!+\!}(\mathbb{R})$, the slash operator $|_\gamma$ acting on the modular form $f$ is defined by
\[f|_{\gamma}(z) := (\det \gamma)^{r/2} \frac{f(\gamma z)}{\big(j_{\gamma}(z)\big)^r}.\]
In particular if $\gamma \in \Gamma$ we have $f|_\gamma = \mu_\gamma f$. More generally, if the group $\gamma^{-1} \Gamma \gamma \cap \slgroup_2(\mathbb{Z})$ has again finite index in $\slgroup_2(\mathbb{Z})$, then $f|_\gamma$ is a modular form of weight $r$ for this group in the sense defined above. Note however that the multiplier system might change even if $\gamma^{-1} \Gamma \gamma = \Gamma$. The finiteness condition is satisfied in particular for any $\gamma \in \slgroup_2(\mathbb{Z})$. The slash operator also satisfies $(f|_\gamma) |_\delta = f|_{\gamma \delta}$ for any two matrices $\gamma, \delta \in \glgroup_2^+(\mathbb{R})$.

In this article we will employ the nonstandard notation $f^\gamma$ to mean the same as $f|_\gamma$ to avoid complications when adding subscripts.

We refer to the rational numbers together with the symbol $\infty$ as cusps. Given a cusp $\mathfrak{a}$ its width $m_\mathfrak{a}$ is the order of the stabilizer of $\mathfrak{a}$ in $\slgroup_2(\mathbb{Z})$ modulo $\Gamma$. A modular form $f$ admits a ``Fourier'' expansion at every cusp $\mathfrak{a}$ (see \cite[\S 4.1]{iwaniec}):
\begin{equation}\label{eq:fourier_expansion}
f^{\sigma_\mathfrak{a}}(z) = \sum_{n \geq 0} a_n^\mathfrak{a} e^{2 \pi i (n+\kappa_\mathfrak{a})z} \qquad (z \in \mathbb{H}).
\end{equation}
In this expression $\sigma_\mathfrak{a}$ stands for a scaling matrix for the cusp $\mathfrak{a}$, which is given by a product $\gamma \eta$ where $\gamma$ is any matrix in $\slgroup_2(\mathbb{Z})$ satisfying $\gamma(\infty) = \mathfrak{a}$ and $\eta = \left(\begin{smallmatrix} \sqrt{m_\mathfrak{a}} & 0 \\ 0 & 1/\sqrt{m_\mathfrak{a}}\end{smallmatrix}\right)$, $m_\mathfrak{a}$ denoting the width of $\mathfrak{a}$. To avoid the ambiguity of choosing between $\gamma$ and $-\gamma$ we will adopt the convention that either $c < 0$, or $c = 0$ and $d = 1$, where $(c, d)$ is the bottom row of $\gamma$. The cusp parameter $0 \leq \kappa_\mathfrak{a} < 1$ appearing in \eqref{eq:fourier_expansion} depends only on $\mathfrak{a}$, while the coefficients $a_n^\mathfrak{a} \in \mathbb{C}$ may assume a finite number of values, as multiplication of $\gamma$ on the right by an unit translation corresponds to the change of variables $z \mapsto z + 1/m_\mathfrak{a}$. Up to this, they are unique. Moreover if one replaces $\mathfrak{a}$ with any other cusp lying in the same orbit modulo $\Gamma$, the right hand side of (\ref{eq:fourier_expansion}) stays invariant up to multiplication by an unimodular constant and the aforementioned translations.

If either $\kappa_\mathfrak{a} > 0$ or $a_0^\mathfrak{a} = 0$ then we say that $f$ is cuspidal at $\mathfrak{a}$, or that $\mathfrak{a}$ is cuspidal for $f$, and we define $f(\mathfrak{a}) = 0$. Otherwise we define $f(\mathfrak{a}) = a_0^\mathfrak{a}$ (note this value depends on the choice of $\sigma_\mathfrak{a}$). Finally, if $f$ is cuspidal at every cusp we say that $f$ is a cusp form. 

We will assume that $\kappa_\mathfrak{a} \in \mathbb{Q}$ for any cusp $\mathfrak{a}$. This is not an important restriction since any modular form coming from an arithmetic setting satisfies this, and most examples are of this kind. Notice that under this assumption (\ref{eq:fourier_expansion}) is, up to a dilation, a Fourier series in the usual sense.

In the case $\mathfrak{a} = \infty$ we may choose $\sigma_\infty = \eta$ and the expansion (\ref{eq:fourier_expansion}) corresponds to
\begin{equation}\label{eq:simple_fourier_exp}
f(mz) = \sum_{n \geq 0} a_n e^{2 \pi i (n+\kappa)z} \quad (z \in \mathbb{H}),
\end{equation}
where $a_n = a^\infty_n$, $m = m_\infty$, $\kappa = \kappa_\infty$. Given $\alpha > 0$ we define the $\alpha$-fractional integral of $f$ as the formal series (\emph{cf.} \cite{chamizo, chamizo_petrykiewicz_ruiz, duistermaat, jaffard})
\begin{equation}\label{eq:fractal_integ}
f_\alpha(mx) := \sum_{n + \kappa > 0} \frac{a_n}{(n+\kappa)^\alpha} e^{2 \pi i (n+\kappa)x}.
\end{equation}
For example, with the notation used in the introduction, $\Im\theta_1(x) = 2\varphi(x)$.

For any $\gamma \in \glgroup_2^+(\mathbb{R})$ such that $\gamma^{-1} \Gamma \gamma \cap \slgroup_2(\mathbb{Z})$ has finite index in $\slgroup_2(\mathbb{Z})$ we may also define $f_\alpha^\gamma := \big(f^{\gamma}\big)_\alpha$. In particular we may always choose $\gamma = \sigma_\mathfrak{a}$, and in this fashion we obtain a collection of related formal series
\[f_\alpha^{\sigma_\mathfrak{a}}(x) = \sum_{n + \kappa_\mathfrak{a} > 0} \frac{a_n^\mathfrak{a}}{(n+\kappa_\mathfrak{a})^\alpha} e^{2 \pi i (n+\kappa_\mathfrak{a})x}.\]
From our previous remarks it follows that $f_\alpha^{\sigma_\mathfrak{a}}$ is uniquely determined by the orbit of the cusp $\mathfrak{a}$ modulo $\Gamma$ up to translation and multiplication by an unimodular constant.

Although the results in this section are stated for an arbitrary modular form, in the proofs (\S\S\ref{sec:lemmas}-\ref{sec:proofs}) we will restrict to the case $m_\infty = 1$, $\kappa_\infty = 0$. This simplifies some arguments and can be assumed without loss of generality, as for any integer $N \geq 1$ the function $f(Nz)$ is again a modular form.

Our first three theorems establish some global and local regularity properties of $f_\alpha$. We use the following notation: for any real $s$ we denote by $[s]$ its integer part, \emph{i.e.}, the biggest integer which is smaller than or equal to $s$, and by $\{s\}$ its decimal part $s - [s]$. We also define $\alpha_0 := r/2$ if $f$ is a cusp form and $\alpha_0 := r$ otherwise.

\begin{theorem}[Global regularity]\label{thm:global}
Let $\alpha > 0$. The following holds: 
\begin{enumerate}
\item
If $\alpha \leq \alpha_0$ the series (\ref{eq:fractal_integ}) defining $f_\alpha$ diverges in a dense set.
\item
If $\alpha > \alpha_0$ the series (\ref{eq:fractal_integ}) defining $f_\alpha$ converges uniformly to a continuous function in all the real line. Moreover $f_\alpha \in \mathcal{C}^{[\alpha - \alpha_0], \{\alpha - \alpha_0\}}(\mathbb{R})$ if $\alpha - \alpha_0 \notin \mathbb{Z}$ and $f_\alpha \in \mathcal{C}^{\alpha - \alpha_0 - 1, 1}_{\log}(\mathbb{R})$ otherwise.
\item
If $0 < \alpha - \alpha_0 \leq 1$ then $f_\alpha \notin \mathcal{C}^{1, 0}(I)$ for any open interval $I$. The same is true for $\Re f_\alpha$ and $\Im f_\alpha$.
\end{enumerate}
\end{theorem}

The statements in theorem~\ref{thm:global} concerning the convergence or divergence of the series (\ref{eq:fractal_integ}) are contained in proposition~3.1 of \cite{chamizo}. Part three is a generalization of lemma~3.5 of \cite{chamizo_petrykiewicz_ruiz}.

For the remaining results stated in this section we will assume $\alpha > \alpha_0$.

\begin{theorem}[Local regularity at rationals]\label{thm:local_rat}
Let $x$ be a rational number and $\beta(x), \beta^{*}(x)$ and $\beta^{**}(x)$ the Hölder exponents of either $f_\alpha$, $\Re f_\alpha$ or $\Im f_\alpha$. Then:
\begin{enumerate}
\item
If $f$ is cuspidal at $x$ then $\beta(x) = 2\alpha - r$. Otherwise $\beta(x) = \alpha - r$.
\item
If $f$ is a cusp form then
\[\beta^*(x) = [\alpha - r/2] + \min\big(1, 2\{\alpha - r/2\}\big).\]
If $f$ is not a cusp form then
\[\beta^*(x) = \begin{cases}
[\alpha - r] + \min\big(1, 2\{\alpha - r\} + r\big) & f \text{ cuspidal at } x, \alpha -r \notin \mathbb{Z} \\
\alpha - r & f \text{ not cuspidal at } x \text{ or } \alpha - r \in \mathbb{Z}.
\end{cases}\]
\item
In any case $\beta^{**}(x) = \alpha - \alpha_0$.
\item
If $0 < \alpha - \alpha_0 \leq 1$ then $f_\alpha$ (resp. $\Re f_\alpha$, $\Im f_\alpha\,$) is not differentiable at any rational point which is not cuspidal for $f$. If $x$ is cuspidal for $f$ then $f_\alpha$ (resp. $\Re f_\alpha$, $\Im f_\alpha\,$) is differentiable at $x$ if and only if $\alpha > (r+1)/2$, and in this case the derivative is given by
\[f_\alpha'(x) = \frac{(2 \pi)^\alpha}{(im)^\alpha \Gamma(\alpha)} \int_{(x)} (z-x)^{\alpha - 1} f'(z)\,dz,\]
where $(x)$ denotes the vertical ray connecting $x$ with $i\infty$.
\end{enumerate}
\end{theorem}

Our previous knowledge on these Hölder exponents at rational points was very poor, specially in the non-cuspidal case (\emph{cf.} theorems~3.3, 3.4 and 3.6 of \cite{chamizo_petrykiewicz_ruiz}). Part $4$ of theorem~\ref{thm:local_rat} is essentially contained in theorem~2.2 of \cite{chamizo}.

The regularity at irrational points depends on how well these points can be approximated by rationals which are not cuspidal for $f$. This is precisely measured by the following quantity:\footnote{The symbol $\ll$ could be replaced by $\leq$ in this definition without affecting the value of $\tau_x$, but this convention simplifies some arguments later on.}
\begin{equation}\label{eq:tau_x}
\tau_x := \sup \left\{\tau : \left|x - \frac{p}{q}\right| \ll \frac{1}{q^\tau} \text{ for infinitely many non-cuspidal rationals } \frac{p}{q} \right\}.
\end{equation}
Note that the inequality $\tau_x \geq 2$ is always satisfied for any irrational number $x$ and, in fact, the number $2$ is always contained in the set on the right hand side of (\ref{eq:tau_x}). This follows from Hedlund's lemma (see \cite[\S3]{patterson}).

\begin{theorem}[Local regularity at irrationals]\label{thm:local_irrat}
Let $x$ be any irrational number and $\beta(x), \beta^{*}(x)$ and $\beta^{**}(x)$ the Hölder exponents of either $f_\alpha$, $\Re f_\alpha$ or $\Im f_\alpha$. Then:
\begin{enumerate}
\item
If $f$ is a cusp form then $\beta(x) = \beta^{*}(x) = \beta^{**}(x) = \alpha - r/2$.
\item
If $f$ is not a cusp form,
\begin{align*}
\beta(x) &= \alpha - \left(1- \frac{1}{\tau_x}\right)r\\
\beta^*(x) &= \begin{cases}
[\alpha - r] + \min\big(1, \{\alpha - r\} + r/\tau_x\big) & \alpha - r \notin \mathbb{Z} \\
\alpha - r & \alpha - r \in \mathbb{Z} 
\end{cases} \\
\beta^{**}(x) &= \alpha - r.
\end{align*}
\end{enumerate}
\end{theorem}

\begin{remark}
Regarding the differentiability of these functions at irrational points we could not prove anything beside the obvious results: it cannot differentiable whenever $\beta(x) < 1$, while it must be for $\beta(x) > 1$.
\end{remark}

The cuspidal case of theorem~\ref{thm:local_irrat} was covered by theorem~3.1 of \cite{chamizo_petrykiewicz_ruiz}, while the non-cuspidal case was previously only known for ``Riemann's example'' (see \cite{jaffard}).

As mentioned in the introduction an approximate functional equation plays a key role in the proof of some of these results. This equation has interest on its own and for this reason we state it here:

\begin{theorem}[Approximate functional equation]\label{thm:func_eq}
Let $\sigma = \sigma_{x_0}$ be any scaling matrix for the cusp $x_0 \in \mathbb{Q}$. Then there exist two nonzero constants $A$, $B$ with $B > 0$ such that:
\[f_\alpha(x) = A i^{-\alpha} f(x_0)\psi(x-x_0) + B |x-x_0|^{2\alpha}(x-x_0)^{ - r} f_\alpha^\sigma\big(\sigma^{-1}x\big) + E(x)\]
where
\[\psi(x) = \begin{cases}
x^{\alpha - r} & \alpha -r \notin \mathbb{Z} \\
x^{\alpha - r}\log{x} & \alpha - r \in \mathbb{Z}
\end{cases}.\]
The error term $E(x)$ lies in the spaces $\mathcal{C}^{1, 0}\big(\mathbb{R} \setminus \{x_0\}\big)$ and $\mathcal{C}^{2\alpha - r + 1}(x_0)$.
\end{theorem}

This theorem remains true for any $\sigma \in \slgroup_2(\mathbb{R})$ as long as $f^\sigma$ is a modular form for a finite index subgroup of $\slgroup_2(\mathbb{Z})$ and the bottom-left entry of $\sigma$ is negative (see \S\ref{sec:approx_eq}). In this case $x_0 = \sigma(\infty)$, and $f(x_0)$ should be understood as the limit of $f^\sigma(z)$ as $\Im{z} \rightarrow \infty$. Note that the theorem may be applied to $f^{\sigma_\mathfrak{b}}$ and $\sigma = \sigma_\mathfrak{b}^{-1} \sigma_\mathfrak{a}$ to show that around some rational $x_0 = \sigma(\infty)$ the graph of the function $f_\alpha^{\sigma_\mathfrak{a}}$ looks like a deformed version of the graph of $f_\alpha^{\sigma_\mathfrak{b}}$. Note also that when $\sigma \in \Gamma$ we have $f(x_0) = 0$ and $f^\sigma = \mu_\sigma f$, and hence the result may be understood as an approximate version of (\ref{eq:modularity}) for $f_\alpha^\sigma$. In this particular case the theorem corresponds to lemma~3.8 of \cite{chamizo_petrykiewicz_ruiz}.

Theorem~\ref{thm:func_eq} was already known in the literature when $f$ is a classical cusp form of integer weight $r > 2$ and $\alpha = r - 1$. In this context $f_{r-1}$ is known as the Eichler integral of $f$ and the approximate equation is in fact exact, the error term corresponding to the period polynomial of $f$ of the Eichler-Shimura theory (\emph{cf.} \cite{eichler}). We recover this result:

\begin{corollary}\label{cor:eichler}
If $f$ is a cusp form of weight $r > 2$ and $\alpha = r-1$ then the error term $E(x)$ in theorem~\ref{thm:func_eq} is given by
\[E(x) = \frac{(2 \pi)^{r-1}}{(im)^{r-1} \Gamma(r-1)} \int_{(x_0)} (z-x)^{r - 2} f(z)\,dz.\]
If moreover $r$ is an integer then $E$ is a polynomial.
\end{corollary}

Theorem~\ref{thm:local_irrat} shows that when $f$ is not a cusp form the pointwise Hölder exponent $\beta$ of $f_\alpha$ at the irrational numbers ranges in a continuum between the values $\alpha - r$ and $\alpha - r/2$. An interesting concept to study in this case is that of the spectrum of singularities, which is defined as the function $d: [0, +\infty) \rightarrow [0, 1]\cup\{-\infty\}$ associating to each $\delta \geq 0$ the Hausdorff dimension of the set $\{x : \beta(x) = \delta\}$ if this set is nonempty and $-\infty$ otherwise (\emph{cf.} \cite{chamizo_petrykiewicz_ruiz, jaffard}). If the image of $d$ is not discrete then it is said that $f_\alpha$ is a multifractal function.

\begin{theorem}[Spectrum of singularities]\label{thm:spectrum}
Let $d$ be the spectrum of singularities of either $f_\alpha$, $\Re f_\alpha$ or $\Im f_\alpha$. Then:
\begin{enumerate}
\item
If $f$ is a cusp form:
\[d(\delta) = \begin{cases}
1 & \delta = \alpha - r/2 \\
0 & \delta = 2\alpha - r \\
-\infty & \text{in other cases}.
\end{cases}\]
\item
If $f$ is not a cusp form:
\[d(\delta) = \begin{cases}
2 + 2\frac{\delta - \alpha}{r} & \alpha - r \leq \delta \leq \alpha - r/2 \\
0 & \delta = 2\alpha - r \\
-\infty & \text{in other cases}.
\end{cases}\]
\end{enumerate}
\end{theorem}
The functions $f_\alpha$, $\Re f_\alpha$ and $\Im f_\alpha$ are therefore multifractal if and only if $f$ is not a cusp form.

\section{Preliminary lemmas}\label{sec:lemmas}

We include in this section some auxiliary results that will be used later on. The first ones describe some general aspects of the behavior of modular forms and their coefficients.

\begin{lemma}[Expansion at the cusps]\label{lem:cusp_expansion}
Let $p, q$ be coprime integers and $z = x + iy \in \mathbb{H}$. Let $m = m_{p/q}$ be the width of the cusp $p/q$. Suppose the quantity $|qz-p|^2/y$ remains uniformly bounded. Then:
\[f(z) = \frac{f(p/q)}{m^{r/2}\big(qz-p\big)^r} + O\left(y^{-r/2} e^{-Ky|qz-p|^{-2}}\right)\]
for some constant $K > 0$ independent of $p/q$.
\end{lemma}

\begin{proof}
A scaling matrix $\sigma = \sigma_{p/q}$ for the cusp $p/q$ has the following form:
\[\sigma^{-1} = \left( \begin{matrix} * & * \\ q\sqrt{m} & -p\sqrt{m} \end{matrix} \right).\]

Once we have fixed $\delta > 0$, from (\ref{eq:fourier_expansion}) one deduces that for some $K' > 0$,
\[f(\sigma w) = \big(j_{\sigma}(w)\big)^r f(p/q) + O\Big(\big|j_{\sigma}(w)\big|^r e^{-K' \Im{w}}\Big) \qquad (\Im{w} \geq \delta).\]
We now perform the change of variables $\sigma w = z$ and use $j_{\sigma}(w) = \big(j_{\sigma^{-1}}(z)\big)^{-1}$ and $\Im{w} = \Im{z} \big|j_{\sigma^{-1}}(z)\big|^{-2}$, from where the desired expansion follows at once. The constant $K$ may be chosen independent of $p/q$ because there are only finitely many equivalence classes of cusps.
\end{proof}

The condition $|qz-p|^2/y \leq \delta$ that appears in the statement of lemma~\ref{lem:cusp_expansion} has the geometric meaning of imposing that $z$ lies in the circle
\[\left\{z : \left| z - \frac{p}{q} - i\frac{\delta}{2q^2}\right| \leq \frac{\delta}{2q^2}\right\} = \gamma\big(\{\Im{z} \geq \delta^{-1}\}\big), \quad \gamma= \left( \begin{matrix} p & * \\ q & * \end{matrix} \right) \in \slgroup_2(\mathbb{Z})\]
These circles are called generalized Ford circles (Speiser circles). We will denote by $\mathcal{F}_{p/q}$ the generalized Ford circle tangent to $p/q$ with $\delta = 2$, and use the following property: $\bigcup \mathcal{F}_{p/q} \supset \{0 < \Im{z} < 1/2\}$. This is clear from the fact that the usual fundamental domain for $\slgroup_2(\mathbb{Z}) \backslash \mathbb{H}$ is contained in $\{\Im{z} \geq 1/2\}$.\footnote{The standard Ford circles (case $\delta = 1$) are intimately related to Farey sequences and diophantine approximation, as is beautifully explained in \cite{ford}.}

We recall that we have defined $\alpha_0$ as $r/2$ if $f$ is a cusp form and $r$ otherwise.

\begin{lemma}\label{lem:simple_bounds}
The modular form $f$ satisfies
\[f(z) \ll (\Im{z})^{-\alpha_0} \qquad (\Im{z} \rightarrow 0^+).\]
Conversely if $x$ is a fixed irrational number one has $f(x+iy) \gg y^{-r/2}$ for infinitely many values of $y \rightarrow 0^+$. If $x$ is a fixed non-cuspidal rational point then one has $f(x+iy) \gg y^{-r}$ for infinitely many values of $y \rightarrow 0^+$.

\end{lemma}

\begin{proof}
Let $z = x+iy$ with $0 < y < 1/2$. By the previous remarks $z$ must be contained in some Ford circle $\mathcal{F}_{p/q}$, and hence lemma~\ref{lem:cusp_expansion} shows that $f(z) \ll y^{-r}$.  If $f$ is cuspidal at $p/q$ then by the same argument $f(z) \ll y^{-r/2}$; in particular this happens when $f$ is a cusp form. 

If $x$ is a non-cuspidal rational point then by lemma~\ref{lem:cusp_expansion} applied at $p/q = x$ we obtain $f(z) \gg y^{-r}$ as $y \rightarrow 0^+$.

(Lem. 3.4 of \cite{chamizo}) For the remaining case we consider the function $g(z) = y^{r/2}|f(z)|$. Since the multiplier $\mu_\gamma$ in (\ref{eq:modularity}) is unimodular it readily follows that $g$ is $\Gamma$-invariant. Now if $x$ is an irrational number then the line $\{\Re{z} = x\}$ cuts the boundary of infinitely many generalized Ford circles for any $\delta \geq 2$ at a sequence of points $x + iy_n$ with arbitrarily small $y_n$. For each of them we may find an element $\gamma \in \slgroup_2(\mathbb{Z})$ sending $x + iy_n$ to a point $w_n$ in the line $\{\Im{w} = \delta^{-1}\}$. We may further assume $-1/2 \leq \Re{w_n} \leq 1/2$ by composing $\gamma$ with a translation if necessary. The inverse of $\gamma$ can be decomposed as $\gamma^{-1} = \gamma' \gamma_i$, where $\gamma' \in \Gamma$ and $\gamma_i$ pertains to a fixed finite right transversal. Therefore $g(x+iy_n) = g(\gamma_i(w_n))$, for some point $w_n$ in the segment $I_\delta = \{\Im{w} = \delta^{ -1}, -1/2 \leq \Re{w} \leq 1/2\}$. Since each of the functions $g_i(z) = g\big(\gamma_i(z)\big)$ has a finite number of zeros in every compact subset of $\mathbb{H}$, we may choose $\delta$ so that none of the $g_i$ vanish on $I_\delta$, guaranteeing $g_i(w_n) \gg 1$. This proves $f(x+iy_n) \gg y_n^{-r/2}$.
\end{proof}

\begin{lemma}\label{lem:irrat_bounds}
Let $\tau \geq 2$ and $x_0$ an irrational number. The following holds:
\begin{enumerate}
\item
If all the non-cuspidal rationals $p/q$ satisfy
\begin{equation}\label{eq:dioph_1}
\left|x_0 - \frac{p}{q}\right| \gg \frac{1}{q^\tau}
\end{equation}
then $f(x+iy) \ll y^{-\left(1 - \frac{1}{\tau}\right)r} + y^{-r}|x-x_0|^{\frac{r}{\tau}}$ for $0 < y < 1/2$.
\item
If there are infinitely many non-cuspidal rationals $p/q$ satisfying
\begin{equation}\label{eq:dioph_2}
\left|x_0 - \frac{p}{q}\right| \ll \frac{1}{q^\tau}
\end{equation}
then $f(x_0 + iy) \gg y^{-\left(1 - \frac{1}{\tau}\right)r}$ for infinitely many values of $y \rightarrow 0^+$.
\end{enumerate}
\end{lemma}

\begin{proof}
1) Let $z = x + iy$ with $0 < y < 1/2$. Then $z$ must be contained in one of the circles $\mathcal{F}_{p/q}$. We will use again the expansion at the cusp given by lemma~\ref{lem:cusp_expansion}. If $p/q$ is cuspidal for $f$ then:
\[f(x+iy) \ll y^{-r/2} \leq y^{-\left(1-\frac{1}{\tau}\right)r}.\]
If $p/q$ is not cuspidal we have
\[f(x+iy) \ll q^{-r} \left( \left(x - \frac{p}{q}\right)^2+y^2\right)^{-r/2}.\]
By hypothesis $p/q$ must satisfy (\ref{eq:dioph_1}) and therefore
\[q^{-r} \ll \left|x_0 - \frac{p}{q}\right|^{r/\tau} \ll \left|x - \frac{p}{q}\right|^{r/\tau} + |x - x_0|^{r/\tau}.\]
Hence:
\[f(x+iy) \ll \left|x - \frac{p}{q}\right|^{r/\tau} \left( \left(x - \frac{p}{q}\right)^2+y^2\right)^{-r/2} + y^{-r} |x - x_0|^{r/\tau}.\]
Arguing by cases depending on whether $y \leq |x-p/q|$ or not it is not hard to prove that the first term is $\ll y^{-\left(1-\frac{1}{\tau}\right)r}$.

2) The case $\tau = 2$ has already been established in lemma~\ref{lem:simple_bounds}, so we may assume $\tau > 2$. By hypothesis there must exist an equivalence class of non-cuspidal rationals modulo $\Gamma$ for which infinitely many satisfy (\ref{eq:dioph_2}). For any of those rationals $p/q$ we choose $z = x_0 + iy$ with $y = q^{-\tau}$ and note that
\[\frac{|qz-p|^2}{y} = q^{2+\tau} \left(\left|x_0 - \frac{p}{q}\right|^2 + y^2\right) \ll q^{2 - \tau}.\]
Applying lemma~\ref{lem:cusp_expansion} again we obtain:
\[|f(x_0 + it)| = C y^{-r/2} \Bigg(\frac{y}{|qz-p|^2}\Bigg)^{r/2} + O\left(y^{-r/2} e^{-Kq^{\tau - 2}}\right) \gg q^{(\tau - 1)r},\]
the constant $C$ not depending on $p/q$. Since $q^{(\tau - 1)r} = y^{-\left(1- \frac{1}{\tau}\right)r}$ this finishes the proof.
\end{proof}

\begin{lemma}\label{lem:partial_sums}
The partial sums in the Fourier expansion (\ref{eq:simple_fourier_exp}) satisfy
\[\sum_{n = 0}^N a_n e^{2 \pi i n x} \ll  N^{\alpha_0}\log{N}.\]
\end{lemma}

\begin{proof}
(Lem. 3.2 of \cite{chamizo}) Using the Dirichlet kernel $D_N(z) = \sum_{|n| \leq N} e^{2 \pi i n z}$ we may write
\[\left| \sum_{n = 0}^N a_n e^{2 \pi i nx} \right| \ll \int_0^1 |f(u + i/N)| |D_N(x-u-i/N)|\, du.\]
We apply lemma~\ref{lem:simple_bounds} to bound the first factor by $N^{\alpha_0}$. Since $\| D_N \|_1 \ll \log{N}$ we obtain the estimate.
\end{proof}

\begin{lemma}\label{lem:coef_growth}
If $f$ is a cusp form then the coefficients in the expansion (\ref{eq:simple_fourier_exp}) satisfy
\[C_1 N^r \leq \sum_{n \leq N} |a_n|^2 \leq C_2 N^r\]
for some $C_1, C_2 > 0$.
\end{lemma}

\begin{proof}
(Lem. 3.2 of \cite{chamizo}) Let $g(z) = (\Im{z})^{r/2}|f(z)|$. We claim that there are constants $C, C' > 0$ such that $|\{x : g(x+i/N) > C\} \cap [0,1]| > C'$ for every integer $N \geq 0$. Because $f$ is cuspidal the function $g$, being $\Gamma$-invariant and bounded in a fundamental domain for $\Gamma \backslash \mathbb{H}$ must be globally bounded. Using Parseval's identity,
\[N^r \ll N^r \int_0^1 |g(u+i/N)|^2\, du = \sum_n |a_n|^2 e^{-4\pi n/N} \ll N^r.\]
The upper bound implies at once
\[\sum_{n \leq N} |a_n|^2 \ll N^r.\]
On the other hand,
\begin{align*}
\sum_{n \leq KN} |a_n|^2 &\geq \sum_n |a_n|^2 e^{-4\pi n/N} - \sum_{n > KN} |a_n|^2 e^{-4\pi n/N} \\
&\gg N^r - C''e^{-2\pi K} N^r,
\end{align*}
where the sum after the minus sign has been estimated summing by parts and using the upper bound. We may now choose $K$ big enough to finish the proof.

We still have to justify the previous claim. Let $C_1, C_2 > 0$ be constants to be determined later and consider the intervals $|x-p/q| \leq C_2/(qN^{1/2})$ with $C_1 N^{1/2} < q < C_2 N^{1/2}$. For $2C_2^3 < C_1$ these are disjoint and cover a positive portion of the interval $[0, 1]$. Suppose that $z = x + i/N$ with $x$ lying in one of those intervals and let $\gamma \in \slgroup_2(\mathbb{Z})$ satisfying $\gamma(p/q) = \infty$. We may decompose $\gamma^{-1} = \gamma' \gamma_i$, where $\gamma' \in \Gamma$ and $\gamma_i$ lies in a fixed finite right-transversal for $\Gamma$. Hence $g(z) = g_i(\gamma z)$ where $g_i(z) = g(\gamma_i z)$. It can be readily checked that $1/(2C_2^2) \leq \Im(\gamma z) \leq 1/C_1^2$, hence it suffices to show that we may choose $C_1$ and $C_2$ to ensure that every $g_i$ is bounded below in that strip. But this follows from the fact that $g_i(z)/(\Im{z})^{r/2} = |f^{\gamma_i}(z)|$ has a Fourier expansion (\ref{eq:fourier_expansion}).
\end{proof}

The following lemma provides an integral representation for $f_\alpha$.

\begin{lemma}\label{lem:integ_rep}
For $\alpha > \alpha_0$ the series (\ref{eq:fractal_integ}) converges uniformly to a continuous function $f_\alpha$, which admits the following integral representation
\begin{equation}\label{eq:integ_rep_2}
f_\alpha(x) = \frac{(2\pi)^\alpha}{(im)^\alpha \Gamma(\alpha)} \int_{(x)} (z-x)^{\alpha - 1} \big(f(z) - f(\mathfrak{\infty})\big)\, dz.
\end{equation}
\end{lemma}

\begin{proof}
Summing by parts (\ref{eq:fractal_integ}) and using the estimates for partial sums given in lemma~\ref{lem:partial_sums} it is plain that the series converges uniformly and hence to a continuous function.

To prove the integral representation we start with
\[f_\alpha(x+iy) = \frac{(2\pi)^\alpha}{m^\alpha\Gamma(\alpha)} \int_0^\infty  t^{\alpha - 1} \big(f(x + iy + it) - f(\infty)\big)\, dt,\]
identity that can be obtained from (\ref{eq:simple_fourier_exp}) integrating the series term by term because of the uniform convergence in the region $\Im{z} \geq y$. Now it suffices to take the limit $y \rightarrow 0^+$ on both sides. The left hand side corresponds to the Abel summation of a converging Fourier series, while in the right hand side the dominated convergence theorem applies with the bounds obtained in lemma~\ref{lem:simple_bounds}.
\end{proof}

Our last lemma is a very particular version of the differentiation under the integral sign theorem.

\begin{lemma}\label{lem:diff}
Let $\gamma \in \slgroup_2(\mathbb{R})$ and let $I$ be a bounded open interval whose closure does not contain the pole of $\gamma$. Let $g(z, x)$ be a function continuously differentiable with respect to $x$ in $I$ and analytic for $z \in \mathbb{H}$. Assume moreover that both $g$ and $g_x$ have exponential decay when $\Im{z} \rightarrow +\infty$ in vertical strips, and that for some $\beta > 0$, $\eta > 0$ they satisfy the following estimates when $z \rightarrow \gamma(x)$ uniformly in $x \in I$:
\begin{align*}
g(z, x) &= O\big((z-\gamma x)^{\beta + \eta - 1}(\Im{z})^{-\eta}\big) \\
g_x(z, x) &= O\big((z-\gamma x)^{\beta + \eta - 2}(\Im{z})^{-\eta}\big)
\end{align*}
Then the function
\[F(x) = \int_{(\gamma x)} g(z, x) \, dz \qquad (x \in I)\]
is in $\Lambda^{\beta}(I)$ for $0 < \beta < 1$, in $\Lambda^1_{\log}(I)$ for $\beta = 1$ and in $\mathcal{C}^{1, 0}(I)$ for $\beta > 1$. In this last case,
\[F'(x) = \int_{(\gamma x)} g_x(z, x) \, dz \qquad (x \in I).\]
\end{lemma}

\begin{proof}
Assume $x \in I$ and $h \neq 0$ satisfying $x+h \in I$. Using Cauchy's theorem together with the estimates for $g$ we can write for $0 < u < v$:
\begin{align*}\label{eq:diff_1}
F(x+h) - F(x) &=  \int_{\gamma(x) + iu}^{\gamma(x) + iv} \big(g(z, x+h) - g(z, x)\big) \, dz \\
&+ O\left(e^{-Kv} + u^\beta + hu^{\beta-1} + \frac{h^{\beta + \eta}}{u^{\eta}}\right). \nonumber
\end{align*}
It is clear now that $F$ must be continuous, as for each $\varepsilon$ we may choose $u$ and $v$ so that for $h$ small enough $|F(x+h) - F(x)| \leq \varepsilon$.

For the rest of the proof we choose $u = h$ and $v = +\infty$, so that the error term is of the form $O\big(h^\beta\big)$. By the mean value theorem:
\[|F(x+h) - F(x)| \ll  h\int_{\gamma(x) + ih}^{\gamma(x) + i\infty} \left|g_x(z, x_z) \right| \, |dz| + O\big(h^\beta\big).\]
Using the estimates for $g_x$ this last integral is of order $O\big(h^{\beta - 1}\big)$ for $0 < \beta < 1$ and of order $O(\log h)$ for $\beta = 1$. 

Suppose now that $\beta > 1$. The estimates for $g_x$ justify the use of the dominated convergence theorem, proving the existence and the formula for $F'$. Finally, the argument used to prove that $F$ is continuous can be applied directly to $F'$ substituting $\beta$ by $\beta - 1$ to conclude that $F'$ is also continuous.
\end{proof}

\section{Approximate functional equation}\label{sec:approx_eq}

Throughout this section we use the following notation: $\sigma$ stands for a fixed matrix in $\slgroup_2(\mathbb{R})$ whose bottom-left is negative and such that $f^\sigma$ is a modular form for a finite index subgroup of $\slgroup_2(\mathbb{Z})$, $x$ will denote an arbitrary real number different from $x_0 = \sigma(\infty)$ and $C_0 = (2\pi)^\alpha/\big(i^\alpha \Gamma(\alpha)\big)$.

To avoid unnecessary distractions we will hide some extra terms that appear during the subsequent manipulations inside the symbol $(\cdots)$; we will deal with them afterwards. The reader can check that all the missing terms appear in (\ref{eq:func_error_1}-\ref{eq:func_error_4}).

Splitting the integral on the right hand side of (\ref{eq:integ_rep_2}) and performing the change of variables $z = \sigma w$ we have:

\begin{align*}
f_\alpha(x) &= C_0 \int_x^{x+2i} (z-x)^{\alpha - 1} f(z)\, dz + (\cdots) \\
&= C_0 \int_S (\sigma w - x)^{\alpha - 1} \big(j_{\sigma}(w)\big)^{r-2} \big(f^\sigma(w) - f(x_0)\big)\, dw + (\cdots).
\end{align*}
where $S$ corresponds to a subarc of the halfcircle with endpoints $\sigma^{-1}(x)$ and $\sigma^{-1}(\infty)$. The term $f(x_0)$ has to be understood as the limit of $f^\sigma(z)$ when $\Im{z} \rightarrow \infty$, definition which agrees with the one given in \S\ref{sec:results} when $\sigma$ is a scaling matrix.

The integrand in the last equation has exponential decay when $\Im{w} \rightarrow +\infty$. Applying Cauchy's theorem to replace $S$ with two vertical rays starting at the endpoints of $S$ and projecting to $i\infty$:
\[f_\alpha(x) = C_0 \int_{(\sigma^{-1}x)} (\sigma w - x)^{\alpha - 1} \big(j_{\sigma}(w)\big)^{r-2} \big(f^\sigma(w) - f(x_0)\big)\, dw + (\cdots).\]
Let $C_1$ denote the constant $C_0 e^{-2 \pi i \alpha}$ if $x < x_0$ and $C_0$ otherwise. Substituting the relation $(\sigma w - x)j_{\sigma}(w) = (w-\sigma^{-1}x)j_{\sigma^{-1}}(x)$ \cite[(2.4)]{iwaniec}:
\[f_\alpha(x) = C_1 \big(j_{\sigma^{-1}}(x)\big)^{\alpha - 1} \int_{(\sigma^{-1}x)} (w - \sigma^{-1} x)^{\alpha - 1} \big(j_{\sigma}(w)\big)^{r-\alpha - 1} \big(f^\sigma(w) - f(x_0)\big)\, dw + (\cdots).\]
Let $\phi(w) = \big(j_{\sigma}(w)\big)^{r-\alpha - 1}$ and denote by $\phi(\sigma^{-1}x^+)$ the limit of $\phi(w)$ when $w \rightarrow \sigma^{-1}x$ from the upper half-plane. Adding and subtracting $\phi(\sigma^{-1}x^+) = \big(j_{\sigma^{-1}}(x)\big)^{\alpha - r + 1}$ and using that $j_{\sigma^{-1}}(x) = B_1 (x-x_0)$ for some constant $B_1 > 0$, we arrive to
\[f_\alpha(x) = B |x-x_0|^{2\alpha} (x-x_0)^{-r} f^\sigma_\alpha(\sigma^{-1}x) + (\cdots).\]

The terms we have omitted so far are the following ones:
\begin{align}
(\cdots) &= -C_0 \frac{(2i)^\alpha}{\alpha} f(\infty) + C_0 \int_{x+2i}^{x+i\infty} (z-x)^{\alpha - 1} \big(f(z) - f(\infty)\big)\, dz \label{eq:func_error_1}\\
&\quad+ C_0 f(x_0) \int_x^{x+2i}(z-x)^{\alpha - 1}\big(j_{\sigma^{-1}}(z)\big)^{-r}\, dz \label{eq:func_error_2}\\
&\quad+ C_0 \Bigg(\int_{x_0}^{x_0 + 2i} + \int_{x_0 + 2i}^{x+2i}\Bigg) (z-x)^{\alpha - 1} \left( f(z) - \frac{f(x_0)}{\big(j_{\sigma^{-1}}(z)\big)^r}\right)\, dz \label{eq:func_error_3} \\
&\quad+ C (x-x_0)^{\alpha - 1} \int_{(\sigma^{-1}x)} (w - \sigma^{-1} x)^{\alpha - 1} \big(\phi(w) - \phi(\sigma^{-1}x^+)\big) \big(f^\sigma(w) - f(x_0)\big)\, dw. \label{eq:func_error_4}
\end{align}
The terms (\ref{eq:func_error_1}) and (\ref{eq:func_error_3}) make sense for any $x \in \mathbb{R}$ and are infinitely many times differentiable with respect to this variable. The other ones are studied in the following lemmas, which complete the proof of theorem~\ref{thm:func_eq}:

\begin{lemma}\label{lem:func_2}
The term (\ref{eq:func_error_2}) admits an expansion of the form:
\[(\ref{eq:func_error_2}) = A i^{-\alpha} f(x_0)\psi(x-x_0) + E(x)\]
where $\psi$ is as in the statement of theorem~\ref{thm:func_eq}.
The constant $A$ is real and nonzero and the error term $E(x)$ is infinitely many times differentiable.
\end{lemma}

\begin{lemma}\label{lem:func_4}
The term (\ref{eq:func_error_4}) lies both in $\mathcal{C}^{1, 0}\big(\mathbb{R} \setminus \{x_0\}\big)$ and in the class $O\big(|x-x_0|^{2\alpha -r + 1}\big)$ when $x \rightarrow x_0$.
\end{lemma}

\begin{proof}[Proof of lemma~\ref{lem:func_2}]
We may assume that $f$ is not cuspidal at $x_0$, since otherwise (\ref{eq:func_error_2}) is equal to zero. Note that in this case by hypothesis $\alpha > r$. Renaming $x-x_0$ to $x$ if necessary we may further assume $x_0 = 0$. Hence up to a nonzero constant of the form $A i^{-\alpha} f(x_0)$ we have to expand asymptotically the function
\begin{equation}\label{eq:func_1}
g(x) = \int_0^{2i} \frac{z^{\alpha - 1}}{(x + z)^{r}}\, dz.
\end{equation}
We will suppose for the moment that $x > 0$ and $\alpha - r \notin \mathbb{Z}$. We have
\[g(x) = x^{-r} \int_0^{2xi} \frac{z^{\alpha-1}}{\Big(1+\frac{z}{x}\Big)^r}\, dz + \int_{2xi}^{2i} \frac{z^{\alpha - r - 1}}{\Big(1+ \frac{x}{z}\Big)^r}\,dz.\]
In the first integral we perform a linear change of variables, while in the second one we substitute the Laurent expansion
\[\left(1+\frac{x}{z}\right)^{-r} = \sum_{k \geq 0} \binom{-r}{k} x^k z^{-k}\]
which is uniformly convergent in the region $|z| \geq 2x$. Integrating term by term the expression now results
\begin{align}
g(x) &= x^{\alpha-r}\int_0^{2i} \frac{z^{\alpha-1}}{(1+z)^r}\, dz + \sum_{k \geq 0} \binom{-r}{k} \frac{x^k}{\alpha - r - k} z^{\alpha - r - k} \Bigg|_{2xi}^{2i} \nonumber \\
&= x^{\alpha-r} \left( \int_0^{2i} \frac{z^{\alpha-1}}{(1+z)^r}\, dz - \sum_{k \geq 0} \binom{-r}{k} \frac{ (2i)^{\alpha - r - k}}{\alpha - r - k} \right) + h(x). \label{eq:func_2}
\end{align}
where $h(x)$ is a function given by a power series which converges in a neighborhood of $0$. Notice that the expression within brackets is a constant $A'$ satisfying
\[A' = \int_0^{T} \frac{z^{\alpha-1}}{(1+z)^r}\, dz - \sum_{k \geq 0} \binom{-r}{k} \frac{T^{\alpha - r - k}}{\alpha - r - k}\]
for any complex $T$ with $|T| > 1$ and  $\arg{T} \neq \pi$: the right hand side is indeed constant as can be easily checked by differentiating with respect to $T$. Hence
\begin{align*}
A' &= \lim_{T \rightarrow +\infty} \left( \int_0^{T} \frac{t^{\alpha-1}}{(1+t)^r}\, dt - \sum_{0 \leq k < \alpha - r} \binom{-r}{k} \frac{T^{\alpha - r - k}}{\alpha - r - k}\right) \\
&= \int_0^\infty t^{\alpha - 1} \left( \frac{1}{(1+t)^r} - \sum_{0 \leq k < \alpha - r}  \binom{-r}{k} \frac{1}{t^{r+k}}\right)\, dt.
\end{align*}
The sum corresponds to the Taylor expansion of order $[\alpha - r]$ of the function $(1-\xi)^{-r}$ multiplied by $\xi^r$ and evaluated at $\xi = 1/t$. Since all the derivatives of this function have constant sign for $\xi > 0$ we deduce $A' \neq 0$. Although the exact value of $A'$ is unimportant, using the integral formula for the error term in the Taylor expansion one can easily obtain a closed formula in terms of beta functions.

Suppose now that $\alpha - r$ is an integer. The same argument can be carried on, but when integrating the Laurent series term by term the term corresponding to $k = \alpha - r$ is now transformed into a logarithm. This term results
\[\binom{-r}{\alpha - r} x^{\alpha - r} \log{z} \Bigg|_{2xi}^{2i} = \binom{-r}{\alpha - r} x^{\alpha - r} \big(-\log(x/i) + \log{2} - \log{T}\big) \quad (T = 2i).\]
The first summand corresponds to the main term, while the other two should be merged into $A'$. This is relevant, as we will need $A' \in \mathbb{R}$ in order to handle the case $x < 0$. We may replace (\ref{eq:func_2}) with: 
\begin{equation}\label{eq:func_3}
g(x) = -\binom{-r}{\alpha - r} x^{\alpha - r} \log(x/i) + A' x^{\alpha - r} + h(x).
\end{equation}

Finally if $x < 0$, we go back to (\ref{eq:func_1}) and notice that
\[g(x) = (-1)^{\alpha - r} \overline{g(-x)},\]
and the very same equation is also satisfied by the main and error terms in equations (\ref{eq:func_2}-\ref{eq:func_3}). Therefore we may apply the results we have obtained for $x > 0$.
\end{proof}

\begin{proof}[Proof of lemma~\ref{lem:func_4}]
Because of the extra cancelation as $w \rightarrow \sigma^{-1}x$ provided by the second factor inside the integral in (\ref{eq:func_error_4}) and the exponential decay given by the third factor when $\Im{z} \rightarrow +\infty$, lemma~\ref{lem:diff} can be applied with $\eta = \alpha_0$ and $\beta + \eta = \alpha + 1$. This shows that (\ref{eq:func_error_4}) is in $\mathcal{C}^{1, 0}\big(\mathbb{R} \setminus \{x_0\}\big)$.

For the second estimate, it suffices to show that
\begin{equation}\label{eq:func_4}
\int_{(\sigma^{-1}x)} (w - \sigma^{-1} x)^{\alpha - 1} \big(\phi(w) - \phi(\sigma^{-1}x^+)\big) \big(f^\sigma(w) - f(x_0)\big)\, dw \ll |x-x_0|^{\alpha - r + 2}
\end{equation}
when $x \rightarrow x_0$. Notice that for $w = \sigma^{-1}x + it$ we have
\[\phi(w) = \big(j_{\sigma}(w)\big)^{r-\alpha - 1} = \left( \frac{1}{(-c)(x-x_0)} + ict\right)^{r - \alpha - 1}\]
where $c$ is the bottom-left entry of $\sigma$. Therefore applying the mean value theorem we obtain for $|x-x_0| \leq 1$:
\[|\phi(w) - \phi(\sigma^{-1}x^+)| \ll \begin{cases}
t|x-x_0|^{\alpha - r + 2} & t \leq |x-x_0|^{-1} \\
t^{r-\alpha-1} & t \geq |x-x_0|^{-1}
\end{cases}.\]
We divide now the integration domain in three intervals and use these estimates, together with the trivial ones for $f^\sigma$, concluding that the left hand side of (\ref{eq:func_4}) is
\begin{align*}
&\ll |x-x_0|^{\alpha - r + 2}\left(\int_0^1 t^\alpha \big(1+ t^{-\alpha_0}\big)\, dt + \int_1^{|x-x_0|^{-1}} t^\alpha e^{-Kt}\, dt\right) \\
&\quad+ \int_{|x-x_0|^{-1}}^\infty t^{r - 2} e^{-Kt}\, dt.
\end{align*}
This proves (\ref{eq:func_4}), since the first two integrals are convergent and the last one has exponential decay when $x \rightarrow x_0$.
\end{proof}

\begin{proof}[Proof of corollary~\ref{cor:eichler}]
If $f$ is a cusp form then (\ref{eq:func_error_2}) and the first summand of (\ref{eq:func_error_1}) vanish. Moreover since $\alpha = r-1$ the function $\phi$ in (\ref{eq:func_error_4}) is constant, and hence this term also vanishes. The remaining terms are:
\begin{align*}
(\cdots) &= C_0 \left(\int_{x_0}^{x_0 + 2i} + \int_{x_0+2i}^{x+2i} + \int_{x+2i}^{x+i\infty}\right) (z-x)^{\alpha - 1} f(z) \, dz \\
&= \frac{(2 \pi)^\alpha}{i^\alpha \Gamma(\alpha)} \int_{(x_0)} (z-x)^{\alpha - 1} f(z) \, dz.\qedhere
\end{align*}
\end{proof}

\section{Wavelet transform}\label{sec:wavelet}

Different definitions of the concept of wavelet can be found in the literature. In this article we consider the following: given $\alpha > 0$, a wavelet is a function $\psi: \mathbb{R} \rightarrow \mathbb{C}$ satisfying:
\begin{enumerate}
\item
$\psi^{(k)}(x) \ll \big(1+|x|\big)^{-\alpha-1}$ for all $k \geq 0$.
\item
$\int_{\mathbb{R}} x^k \psi(x)\, dx = 0$ for $0 \leq k < \alpha$.
\item
Either
\begin{equation*}
\int_0^\infty |\hat{\psi}(\xi)|^2 \, \frac{d\xi}{\xi} = \int_0^\infty |\hat{\psi}(-\xi)|^2\, \frac{d\xi}{\xi} = 1
\end{equation*}
or
\begin{equation*}
\hat{\psi}(\xi) = 0 \text{ if } \xi < 0 \quad \text{and} \quad \int_0^\infty |\hat{\psi}(\xi)|^2 \, \frac{d\xi}{\xi} = 1.
\end{equation*}
\end{enumerate}
These axioms are adapted from \cite[\S2]{jaffard2}. The differences with the definition employed by Jaffard are subtle but important, and will allow us to avoid the very unnatural hypothesis that appear in the main theorems of \cite{chamizo_petrykiewicz_ruiz}. We also define the wavelet transform of a bounded function $f$ with respect to the wavelet $\psi$ as
\[W(a, b) := \frac{1}{a} \int_{\mathbb{R}} f(t) \bar{\psi}\left(\frac{t-b}{a}\right)\, dt \quad (b \in \mathbb{R}, a > 0).\]

If we also ask $f$ to be periodic, with vanishing integral on each period, and satisfying $\hat{f}(\xi) = 0$ for $\xi < 0$ in the distributional sense in case the same is satisfied by $\psi$, then the following inversion formula holds:
\begin{equation}\label{eq:wavelet_inversion}
f(x) = \int_{\mathbb{R}^{+} \int_\mathbb{R}} W(a, b) \psi\left(\frac{x-b}{a}\right) \, \frac{db \, da}{a^2}.
\end{equation}
The proof of this fact can be found in \cite{holschneider_tchamitchian}. The outer integral in (\ref{eq:wavelet_inversion}) in principle has to be understood as an improper Riemann integral, but in our applications it will be absolutely convergent.

The wavelet transform allows us to reformulate questions concerning the regularity of $f$ in a point $x_0$ as questions about the growth of its wavelet transform $W$ in a neighborhood of the corresponding point $(0^+, x_0)$, as it is shown in the following two theorems:
\begin{theorem}\label{thm:wavelet_1}
Let $0 < \beta < \alpha$. If $f \in \mathcal{C}^\beta(x_0)$ then
\[W(a, b) \ll a^\beta + |b-x_0|^\beta\]
when $(a, b) \rightarrow (0^+, x_0)$.
\end{theorem}
\begin{theorem}\label{thm:wavelet_2}
Let $0 < \beta' < \beta < \alpha$. If
\[W(a, b) \ll a^\beta + a^{\beta - \beta'} |b-x_0|^{\beta'}\]
when $(a, b) \rightarrow (0^+, x_0)$ then $f \in \mathcal{C}^\beta(x_0)$ if $\beta$ is not an integer and $f \in \mathcal{C}^\beta_{\log}(x_0)$ otherwise.
\end{theorem}
The bounds involving $W(a, b)$ in these two theorems may also be written in the forms $a^\beta\left(1+\frac{|b-x_0|}{a}\right)^\beta$ and $a^\beta \left(1+\frac{|b-x_0|}{a}\right)^{\beta'}$, respectively, from where it is clear that the second one constitutes a strengthening of the first.

\begin{remark}
The last two theorems are analogous to proposition~1 of \cite{jaffard2} for our definition of wavelet. With our notation, the use of the definition given in \cite{jaffard2} would require the extra hypothesis $[\beta] \leq [\alpha] - 1$. Note also that the logarithm appearing when $\beta \in \mathbb{Z}$ is neglected in \cite{jaffard2} (and the proof for $\beta \geq 1$ left to the reader). Indeed, theorem~\ref{thm:func_eq} shows that there are examples for which the logarithm is necessary (\emph{cf.} \S\ref{sec:proofs}).
\end{remark}

\begin{proof}[Proof of theorem~\ref{thm:wavelet_1}]
We can assume without loss of generality $x_0 = 0$. By hypothesis there is a polynomial $P$ of degree strictly smaller than $\alpha$ such that
\[|f(x) - P(x)| \ll |x|^\beta,\]
estimate which we may assume to hold globally. Hence, by the property~2 of analytic wavelets,
\begin{align*}
W(a, b) &\ll \frac{1}{a} \int_{\mathbb{R}} |f(t) - P(t)| \left| \psi\left(\frac{t-b}{a}\right)\right| \, dt \\
&\ll \frac{1}{a} \int_{\mathbb{R}} \frac{|t|^\beta}{\left(\left|\frac{t-b}{a}\right| + 1\right)^{\alpha + 1}} \, dt \\
&\ll a^\beta \int_{\mathbb{R}}\frac{|t|^\beta}{\big(|t| + 1\big)^{\alpha + 1}}\, dt + |b|^\beta \int_{\mathbb{R}} \frac{dt}{\big(|t| + 1\big)^{\alpha + 1}}\\
&\ll a^\beta + |b|^\beta.\qedhere
\end{align*}\end{proof}

In order to prove theorem~\ref{thm:wavelet_2} we shall use the inversion formula (\ref{eq:wavelet_inversion}), which for convenience will be written in the following way:
\begin{equation}\label{eq:wavelet_inversion_2}
f(x) = \int_{\mathbb{R}^+} \omega(a, x) \, \frac{da}{a}
\end{equation}
where
\begin{equation}\label{eq:wavelet_omega}
\omega(a, x) = \frac{1}{a} \int_{\mathbb{R}} W(a, b) \psi\left(\frac{x-b}{a}\right)\, db.
\end{equation}
We prove first some estimates for $\omega$. In particular they show that the integral in (\ref{eq:wavelet_inversion_2}) is absolutely convergent.
\begin{lemma}\label{lem:omega_bounds}
Under the hypothesis of theorem~\ref{thm:wavelet_2} the function $x \mapsto \omega(a, x)$ is infinitely many times differentiable and satisfies for all $k \geq 0$ and for some $\delta > 0$:
\begin{align}
\frac{\partial^k \omega}{\partial x^k}(a,x) &\ll a^{-k-1}, \label{eq:wavelet_omega_bound_1}\\
\frac{\partial^k \omega}{\partial x^k}(a,x) &\ll a^{\beta - k} + a^{\beta - \beta' - k} |x-x_0|^{\beta'} \qquad (a \leq 1, |x-x_0| \leq \delta) \label{eq:wavelet_omega_bound_2}
\end{align}
\end{lemma}

\begin{proof}
It is clear that $W(a, b)$ is uniformly bounded and $\psi$ and all its derivatives have decay (property 1 of analytic wavelets). Therefore we may differentiate (\ref{eq:wavelet_omega}) under the integral sign obtaining
\begin{equation}\label{eq:wavelet_omega_2}
\frac{\partial^k \omega}{\partial x^k} (a, x) = \frac{1}{a^{k+1}} \int_{\mathbb{R}} W(a, b) \psi^{(k)}\left( \frac{x-b}{a} \right) \, db.
\end{equation}

Integrating by parts in the definition of $W(a, b)$ and using that the integral over each period of $f$ vanishes it is readily seen that $W(a, b) \ll a^{-1}$. Plugging this into (\ref{eq:wavelet_omega_2}) one obtains (\ref{eq:wavelet_omega_bound_1}).

To prove (\ref{eq:wavelet_omega_bound_2}) we first assume without loss of generality that $x_0 = 0$, and that the bounds in the statement of theorem~\ref{thm:wavelet_2} hold uniformly in the neighborhood $a \leq 1$ and $|b| \leq 2\delta$. We have for $a \leq 1$ and $|x| \leq \delta$:
\begin{align*}
\frac{\partial^k \omega}{\partial x^k} (a, x) &\ll \frac{1}{a^{k+1}} \int_{|b| \leq 2\delta}  \frac{a^\beta + a^{\beta - \beta'}|b|^{\beta'}}{\left(\left|\frac{x-b}{a}\right|+1\right)^{\alpha + 1 }} \, db + \frac{1}{a^{k+1}} \int_{|b| > 2\delta} \frac{db}{\left(\left|\frac{x-b}{a}\right|+1\right)^{\alpha + 1 }} \\
&\ll a^{\beta - k} + a^{\beta - \beta' - k} \int_{\mathbb{R}} \frac{|x-at|^{\beta'}}{\big(|t| + 1\big)^{\alpha + 1}} \, dt + \frac{1}{a^k} \int_{t > \delta/a} \frac{dt}{(t+1)^{\alpha + 1}} \\
&\ll a^{\beta - k} + a^{\beta - \beta' - k}|x|^{\beta'}.\qedhere
\end{align*}
\end{proof}

\begin{proof}[Proof of theorem~\ref{thm:wavelet_2}]
Again we can assume $x_0 = 0$. Let $N = [\beta]$ if $\beta$ is not an integer and $N = \beta - 1$ otherwise. We perform a Taylor expansion of order $N$ on $\omega$:
\[\omega(a, x) = \sum_{k = 0}^N \frac{\partial^k \omega}{\partial x^k} (a, 0) \frac{x^k}{k!} + E(a, x).\]
Using the bounds of lemma~\ref{lem:omega_bounds} we can plug this into (\ref{eq:wavelet_inversion_2}) to obtain
\[f(x) = P(x) + \int_{\mathbb{R}^+} E(a, x) \, \frac{da}{a}\]
for certain polynomial $P$ of degree at most $[\beta]$. It suffices to prove that the integral term has the right behavior when $x \rightarrow 0$.

We split the integral. In the range $a \leq |x|$ we use (\ref{eq:wavelet_omega_bound_2}) with either $x = 0$ or $k = 0$ to obtain
\[\left|\int_{a \leq |x|} E(a, x)\, \frac{da}{a}\right| \leq \int_{a \leq |x|} |\omega(a, x)| \, \frac{da}{a} + \sum_{k = 0}^N \frac{|x|^k}{k!} \int_{a \leq |x|} \left| \frac{\partial^k \omega}{\partial x^k}(a, 0)\right| \frac{da}{a} \ll |x|^\beta.\]
In the complementary range, assuming that $\beta$ is not an integer, we use the formula for the Taylor error term together with (\ref{eq:wavelet_omega_bound_2}):
\[\left|\int_{a \geq |x|} E(a, x) \, \frac{da}{a}\right| \leq \frac{|x|^{N+1}}{(N+1)!} \int_{a \geq |x|} \left| \frac{\partial^{N+1} \omega}{\partial x^{N+1}}(a, \xi_{a, x})\right| \frac{da}{a} \ll |x|^\beta.\]
When $\beta$ is an integer the same argument works using (\ref{eq:wavelet_omega_bound_2}) in the range $|x| \leq a \leq 1$ and (\ref{eq:wavelet_omega_bound_1}) in the range $a \geq 1$. The right hand side has to be replaced by $|x|^\beta \log|x|$.
\end{proof}

Following \cite{chamizo_petrykiewicz_ruiz, jaffard2} we apply these theorems to $f_\alpha$, where $f$ is a modular form, with $\psi(x) = (x+i)^{-\alpha - 1}$. The reader can easily verify that $\psi$ satisfies properties 1 and 2 of our definition of wavelet. In order to check property 3 we compute $\hat{\psi}$. The integral
\[\hat{\psi}(\xi) = \int_{\mathbb{R}} \frac{e^{-2\pi i \xi x}}{(x+i)^{\alpha + 1}} \, dx\]
vanishes for $\xi \leq 0$ by Cauchy's theorem. For $\xi > 0$ we perform a change of variables obtaining
\[\hat{\psi}(\xi) = \xi^\alpha e^{-2\pi \xi} \int_{\mathbb{R} + \xi i} \frac{e^{-2\pi i z}}{z^{\alpha + 1}} \, dz\]
and by Cauchy's theorem the integral on the right hand side is a constant with respect to $\xi$. The exact value of the constant is not important, since $\psi$ needs not to be normalized for theorems~\ref{thm:wavelet_1} and \ref{thm:wavelet_2} to hold, although it can be explicitly computed by means of Hankel's contour integral for the reciprocal of the gamma function (\emph{cf.} \cite[\S12.22]{whittaker_watson}).

It is also clear that $f_\alpha$ is a periodic function, with vanishing integral on each period, and whose Fourier transform (in the distributional sense) is supported only in the positive frequencies. To compute its wavelet transform with respect to $\psi$ it suffices to compute the one for $g(x) = e^{2 \pi i \lambda x}$. This can be done using some basic properties of the Fourier transform:
\begin{equation}\label{eq:wavelet_g_transform}
W_g(a, b) = e^{2\pi i \lambda b} \bar{\hat{\psi}}(\lambda a) = \begin{cases}
C a^{\alpha} \lambda^\alpha e^{2 \pi i \lambda (b+ai)} & \lambda > 0 \\
0 & \lambda \leq 0.
\end{cases}\end{equation}
Hence
\begin{equation}\label{eq:wavelet_f_transform}
W_{f_\alpha}(a, b) = C a^\alpha \big(f(b+ai) - f(\infty)\big).
\end{equation}

\begin{corollary}\label{cor:wavelet}
If for some $0 < \beta < \alpha$ one has $f_\alpha \in \mathcal{C}^\beta(x_0)$ then
\[f(b+ai) \ll a^{\beta - \alpha} + a^{-\alpha} |b-x_0|^\beta\]
when $(a, b) \rightarrow (0^+, x_0)$. Reciprocally, if for some $0 < \beta' < \beta < \alpha$ one has
\[f(b+ai) \ll a^{\beta - \alpha} + a^{\beta - \beta' - \alpha}|b-x_0|^{\beta'}\]
when $(a, b) \rightarrow (0^+, x_0)$, then $f_\alpha \in \mathcal{C}^\beta(x_0)$ if $\beta$ is not an integer and $f_\alpha \in \mathcal{C}^\beta_{\log}(x_0)$ otherwise. Moreover both statements remain true if one replaces $f_\alpha$ by its real or imaginary parts.
\end{corollary}

\begin{proof}
The part of the theorem concerning $f_\alpha$ follows at once from theorems~\ref{thm:wavelet_1} and \ref{thm:wavelet_2} and (\ref{eq:wavelet_f_transform}). Also note that if $f_\alpha \in \mathcal{C}^\beta(x_0)$ or $f_\alpha \in \mathcal{C}^\beta_{\log}(x_0)$ then the same must hold for the real and the imaginary parts of $f_\alpha$.

On the other hand, $\Re f_\alpha$ and $\Im f_\alpha$ are bounded functions, and hence their wavelet transforms are well defined. By rewriting the sine and cosine functions involved in their Fourier series as sums of exponentials and applying (\ref{eq:wavelet_g_transform}) one obtains
\[W_{f_\alpha}(a, b) = 2 W_{\Re f_\alpha}(a, b) = 2i W_{\Im f_\alpha}(a, b).\]
Since the inversion formula (\ref{eq:wavelet_inversion}) is not used in the proof of theorem~\ref{thm:wavelet_1}, we may apply this theorem to $\Re f_\alpha$ and $\Im f_\alpha$.
\end{proof}

\section{Regularity theorems}\label{sec:proofs}

This section contains the proofs of theorems~\ref{thm:global}, \ref{thm:local_rat} and \ref{thm:local_irrat}.

\begin{lemma}\label{lem:beta_star}
If $f_\alpha$ is in $\mathcal{C}^{k, 0}(x)$ for some $k < \alpha - \alpha_0$, but cannot be continuously differentiated $k+1$ times in any open interval containing $x$, then
\[\beta^{*}(x) = k + \min \{1, \beta_{\alpha - k}(x)\},\]
where $\beta_{\alpha-k}$ denotes the pointwise Hölder exponent of $f_{\alpha - k}$. This formula extends to $\Re f_\alpha$ and $\Im f_\alpha$ if both these functions satisfy the hypothesis and their pointwise Hölder exponents coincide.
\end{lemma}

\begin{proof}
This follows at once from the identity $f_\alpha^{(k)} = (2 \pi i)^k f_{\alpha - k}$ and the definition of $\beta^{*}$.
\end{proof}

In order to prove theorems~\ref{thm:global} and \ref{thm:local_rat} we anticipate two very simple results which will come in handy. Applying corollary~\ref{cor:wavelet} with the bounds from lemma~\ref{lem:simple_bounds} we obtain $\beta(x) = \alpha - r/2$ for $f$ cuspidal and $x$ irrational and $\beta(x) = \alpha - r$ for $f$ not cuspidal and $x$ any non-cuspidal rational.

\begin{proof}[Proof of theorem~\ref{thm:global}]
1) (Proposition~3.1 of \cite{chamizo}) If the series defining $f_\alpha$ converge at a certain point for $\alpha < \alpha_0$ then summing by parts the series defining $f_{\alpha_0}$ must also converge at that point, and therefore we may reduce to this case.

Suppose first that $f$ is cuspidal, we will prove that $f_{r/2}$ diverges at any irrational point $x$. Considering the kernels of summability $\varphi_1(u) = e^{-2\pi u}(u^{r/2}+1)$ and $\varphi_2(u) = e^{-2\pi u}$, we have (see Th. III.1.2 of \cite{zygmund}):
\[\lim_{y \rightarrow 0^+} y^{r/2} f(x+iy) = \lim_{t \rightarrow 0^+} \left(\sum_{n > 0} A_n \varphi_1(ny) - \sum_{n > 0} A_n \varphi_2(ny)\right) = 0\]
with $A_n = \frac{a_n}{n^{r/2}} e^{2\pi i n x}$, as long as $f_{r/2}$ converges at $x$; but this contradicts lemma~\ref{lem:simple_bounds}.

Suppose now that $f$ is not cuspidal. We prove that $f_r$ is not Abel summable at any non-cuspidal rational point $x$. If this were not the case then by lemma~\ref{lem:integ_rep} we would have for some $\ell \in \mathbb{C}$,
\[\ell = \lim_{y \rightarrow 0^+} f_r(x+iy) = \lim_{y \rightarrow 0^+} \frac{(2\pi)^\alpha}{\Gamma(\alpha)} \int_y^{\infty} (t-y)^{r-1} \big(f(x + it) - f(\infty)\big)\, dt.\]
But since by the expansion at the cusp the term $f(x+it)$ behaves like $Ct^{-r}$ for small $t$, the right hand side diverges.

2) The result follows from applying lemma~\ref{lem:diff} to the integral representation given by lemma~\ref{lem:integ_rep} repeatedly.

3) Suppose first that $f$ is not cuspidal. If $\alpha - r < 1$ then neither $f_\alpha$ nor its real or imaginary parts are differentiable at any non-cuspidal rational, since they are at most $(\alpha - r)$-Hölder at these points. Only the limit case $\alpha = r + 1$ remains. But in this case we may appeal to theorem~\ref{thm:func_eq}, since $2\alpha - r = r + 2 > 1$ implies that both the second term and the error term are differentiable at the rational $x_0$, and the first term is not if $x_0$ is non-cuspidal. A more detailed analysis shows that neither the real nor the imaginary parts of the function $Cx\log{x}$ are differentiable at $0$ for any complex constant $C$.

(Lem. 3.7 of \cite{chamizo_petrykiewicz_ruiz}) Suppose now that $f$ is cuspidal. If $f_\alpha$ is in $\mathcal{C}^{1, 0}(I)$ then by theorem~\ref{thm:func_eq} it is also in $\mathcal{C}^{1, 0}(\gamma(I))$ for any $\gamma \in \Gamma$. It follows that $f_\alpha'$ must exist and be continuous everywhere, and by Bessel's inequality
\[\|f_\alpha'\|_2^2 \gg \sum_{n > 0} \frac{|a_n|^2}{n^{2\alpha-2}}.\]
But the right hand side diverges for $\alpha - r/2 \leq 1$ as can be checked by summing by parts and using the estimates of lemma~\ref{lem:coef_growth}.

Finally assume that either $\Re f_\alpha$ or $\Im f_\alpha$ is in $\mathcal{C}^{1, 0}(I)$. Since the periodic Hilbert transform preserves the Sobolev space $H^1$ (\emph{cf.} \S3 of \cite{iorio_iorio}) and sends a Fourier series to its conjugate series (and therefore $\Re f_\alpha$ to $\Im f_\alpha$ and $\Im f_\alpha$ to $-\Re f_\alpha$, \emph{cf.} \S II.5 of \cite{zygmund}), the function $f_\alpha$ must, at least, have a weak derivative in $L^2(I')$ for some smaller interval $I'$. This is enough to carry on the previous argument.
\end{proof}

\begin{proof}[Proof of theorem~\ref{thm:local_rat}]
Let $x_0$ be a rational number. 

1) If $f$ is not cuspidal at $x_0$ then we already know $\beta(x_0) = \alpha - r$. Hence may assume that $f$ is cuspidal at $x_0$. Choose a scaling matrix $\sigma$ satisfying $\sigma(\infty) = x_0$ and apply theorem~\ref{thm:func_eq}. We deduce that $f_\alpha \in \mathcal{C}^{2\alpha - r}(x_0)$ and that $f_\alpha \notin \mathcal{C}^{2\alpha - r + \varepsilon}(x_0)$ for any $\varepsilon > 0$, since the term $\sigma^{-1} x$ diverges to $\infty$ when $x \rightarrow x_0$ and $f_\alpha^\sigma$ is a nonzero periodic function. Hence $\beta(x_0) = 2\alpha - r$. The same must be true for $\Re f_\alpha$ and $\Im f_\alpha$ as long as the image $f_\alpha^\sigma(\mathbb{R})$ is not contained in any one-dimensional subspace of $\mathbb{C}$. This is indeed the case as $f_\alpha^\sigma$ corresponds to a Fourier series with only positive frequencies.

2) The exponent $\beta^{*}$ is determined by applying lemma~\ref{lem:beta_star} with $k = [\alpha - \alpha_0]$ if $\alpha - \alpha_0 \notin \mathbb{Z}$ and $k = \alpha - \alpha_0 - 1$ otherwise (\emph{cf.} theorem~\ref{thm:global}).

3) To determine $\beta^{**}$ note first that theorem~\ref{thm:global} implies $\beta^{**}(x) \geq \alpha - \alpha_0$. Since this exponent also satisfies $\beta^{**}(x) \leq \liminf_{t \rightarrow x} \beta(t)$, as can be readily seen from its definition, and we have $\beta(x) = \alpha - \alpha_0$ for a dense set (the irrational numbers if $f$ is cuspidal and the non-cuspidal rationals otherwise) we conclude $\beta^{**}(x) = \alpha - \alpha_0$ for all $x$.

4) The case $x_0$ non-cuspidal has already been treated in the proof of theorem~\ref{thm:global}, part 3. Hence we may suppose that $f$ is cuspidal at $x_0$. We appeal again to theorem~\ref{thm:func_eq} but now we will use the explicit expression for the error term (\emph{cf.} \S\ref{sec:approx_eq}):
\[f_\alpha(x) = B|x-x_0|^{2\alpha}(x - x_0)^{-r} f_\alpha^\sigma(\sigma^{-1}x) + (\ref{eq:func_error_1}) + (\ref{eq:func_error_3}) + (\ref{eq:func_error_4}).\]
Terms (\ref{eq:func_error_1}) and (\ref{eq:func_error_3}) are everywhere differentiable, while term (\ref{eq:func_error_4}) can be differentiated at $x_0$ by lemma~\ref{lem:func_4}. Hence $f_\alpha$ is differentiable at $x_0$ if and only if the first summand is. Since $f_\alpha^\sigma$ is bounded, nonzero and periodic this will happen if and only if $2\alpha - r > 1$. The same must be true for the real and imaginary parts of $f_\alpha$, since the image of $f_\alpha^\sigma$ is not contained in any one-dimensional subspace of $\mathbb{C}$.

Hence whenever $f_\alpha'(x_0)$ exists it is given by the sum of the derivatives of the terms (\ref{eq:func_error_1}) and (\ref{eq:func_error_3}) evaluated at $x_0$ (the other terms have vanishing derivative at $x_0$). Differentiating under the integral sign and integrating by parts one obtains the desired formula.
\end{proof}

\begin{proof}[Proof of theorem~\ref{thm:local_irrat}]
Let $x_0$ an irrational number. The pointwise Hölder exponent $\beta(x_0)$ is deduced by applying corollary~\ref{cor:wavelet} to the estimates of lemma~\ref{lem:simple_bounds} if $f$ is a cusp form and of lemma~\ref{lem:irrat_bounds} otherwise. The exponent $\beta^{*}(x_0)$ follows from lemma~\ref{lem:beta_star}, while $\beta^{**}(x_0)$ was already determined in the proof of theorem~\ref{thm:local_rat}, part 3.
\end{proof}

\section{Spectrum of singularities}\label{sec:spectrum}

In order to prove theorem~\ref{thm:spectrum} we will need some tools from diophantine analysis. More concretely we will need a refinement of the following classic theorem:

\begin{theorem}[Jarník-Besicovitch]\label{thm:jarnik_b}
Let $\tau \geq 2$. The Hausdorff dimension of the set
\begin{equation}\label{eq:A_tau}
A_\tau := \left\{ x : \left|x - \frac{p}{q}\right| \ll \frac{1}{q^\tau} \text{ for infinitely many rationals } \frac{p}{q}\right\}
\end{equation}
is $2/\tau$. Moreover, if we denote by $\mathcal{H}^t$ the $t$-dimensional outer Hausdorff measure, $\mathcal{H}^{2/\tau}\big(A_\tau\big) = \infty$.
\end{theorem}

For the proof of theorem~\ref{thm:jarnik_b} when $\tau > 2$ we refer the reader to \cite{jarnik}. The case $\tau = 2$ follows from Dirichlet's approximation theorem.

We are going to write $\mathfrak{a} \sim \mathfrak{b}$ to denote that these two cusps lie in the same orbit modulo $\Gamma$, \emph{i.e.}, that $\mathfrak{b} = \gamma(\mathfrak{a})$ for some $\gamma \in \Gamma$. The theorem we need is the following, which takes into account that rational numbers are well distributed among the different classes of cusps.

\begin{theorem}\label{thm:jarnik_b2}
Let $\mathfrak{a}$ be a cusp for $\Gamma$ and $\tau \geq 2$. The Hausdorff dimension of the set
\begin{equation*}
A_\tau^\mathfrak{a} := \left\{ x : \left|x - \frac{p}{q}\right| \ll \frac{1}{q^\tau} \text{ for infinitely many rationals } \frac{p}{q} \sim \mathfrak{a}\right\}
\end{equation*}
is $2/\tau$. Moreover, if we denote by $\mathcal{H}^t$ the $t$-dimensional outer Hausdorff measure, $\mathcal{H}^{2/\tau}\big(A_\tau^\mathfrak{a}\big) = \infty$.
\end{theorem}

Theorem~\ref{thm:jarnik_b2} is a particular case of more general results about Fuchsian groups (\emph{cf.} \cite{velani}). We provide here an elemental proof based on theorem~\ref{thm:jarnik_b}.

\begin{proof}
Note that we may assume without loss of generality that $\Gamma$ is a normal subgroup of $\slgroup_2(\mathbb{Z})$. Indeed, if this is not the case, we simply replace $\Gamma$ with the biggest normal group it contains, \emph{i.e.}, the intersection of all its conjugates. The normality of $\Gamma$ implies that the action of $\slgroup_2(\mathbb{Z})$ on the equivalence classes of cusps modulo $\Gamma$ is well-defined.

Let $\gamma$ be any matrix in $\slgroup_2(\mathbb{Z})$ and $x$ an irrational number in $A_\tau^\mathfrak{a}$. We claim that if $p/q$ is a rational number in a neighborhood of $x$ and $q'$ denotes the denominator of $\gamma(p/q)$ then $q' \ll q$. Indeed $q' = cp+dq$, and $p \ll q$ because $|p/q| \sim |x|$. From this together with the mean value theorem applied to $|\gamma(x) - \gamma(p/q)|$ we deduce that $\gamma(x) \in A_\tau^{\gamma(\mathfrak{a})}$. The argument can also be applied to $\gamma^{-1}$ and therefore:
\begin{equation}\label{eq:A_tau_gamma}
\gamma(A_\tau^\mathfrak{a}) = A_\tau^{\gamma(\mathfrak{a})}.
\end{equation}

For any Lipschitz function $\varphi$ with Lipschitz constant $C$ and any set $\Omega$ we have
\begin{equation}\label{eq:hausdorff_ineq}
\mathcal{H}^t\big(\varphi(\Omega)\big) \leq C^t \mathcal{H}^t(\Omega).
\end{equation}
This follows from the definition of Hausdorff outer measure. We want to apply this to prove that all the sets $A_\tau^\mathfrak{a}$ have roughly the same size when $\mathfrak{a}$ ranges through a set of representatives of the equivalence classes of the cusps modulo $\Gamma$, but the Möbius transformation $\gamma$ is not Lipschitz in any neighborhood of its pole. This problem has a simple workaround. Let $m$ be the width of the cusp $\infty$ and $I$ any interval of length $m$ not containing the pole of $\gamma$, and whose image $J = \gamma(I)$ is also of length $m$. Then from (\ref{eq:A_tau_gamma}) we have
\begin{align*}
\gamma(A_\tau^\mathfrak{a} \cap I) &= A_\tau^{\gamma(\mathfrak{a})} \cap J \\
A_\tau^\mathfrak{a} + m &= A_\tau^\mathfrak{a}.
\end{align*}
Applying (\ref{eq:hausdorff_ineq}),
\[\mathcal{H}^t(A_\tau^{\gamma(\mathfrak{a})}) \ll \mathcal{H}^t(A_\tau^\mathfrak{a}).\]
The opposite inequality is also true and hence the Hausdorff dimension of the set $A_\tau^\mathfrak{a}$ must be independent of $\mathfrak{a}$. Since we also know by theorem~\ref{thm:jarnik_b} that $A_\tau = \bigcup_{\mathfrak{a}} A_\tau^\mathfrak{a}$ has dimension $2/\tau$, we conclude that all the $A_\tau^\mathfrak{a}$ must have exactly that dimension. It is also immediate that $\mathcal{H}^{2/\tau}\big(A_\tau^\mathfrak{a}\big) = \infty$.
\end{proof}

\begin{corollary}\label{cor:spec_hausdorff}
Let $2 \leq \tau \leq +\infty$. The Hausdorff dimension of the set $\{x : \tau_x = \tau\}$ is $2/\tau$.
\end{corollary}

For the definition of $\tau_x$ see (\ref{eq:tau_x}).

\begin{proof}
Assume $\tau > 2$ and let $\Xi$ be a set of representatives of the equivalence classes of cusps at which $f$ is not cuspidal. We have the identity
\[\{x : \tau_x = \tau\} = \bigcap_{\tau' < \tau} \bigcup_{\mathfrak{a} \in \Xi} A_{\tau'}^\mathfrak{a} \setminus \bigcup_{\tau' > \tau} \bigcup_{\mathfrak{a} \in \Xi} A_{\tau'}^\mathfrak{a}.\]
By theorem~\ref{thm:jarnik_b2} the set on the right hand side has Hausdorff dimension at most $2/\tau$. On the other hand from the same theorem one deduces that for $\tau < +\infty$ we have
\[\mathcal{H}^{2/\tau}\bigg(\bigcap_{\tau' < \tau} \bigcup_{\mathfrak{a} \in \Xi} A_{\tau'}^\mathfrak{a}\bigg) = \infty, \quad \quad \mathcal{H}^{2/\tau}\bigg(\bigcup_{\tau' > \tau} \bigcup_{\mathfrak{a} \in \Xi} A_{\tau'}^\mathfrak{a}\bigg) = 0.\]
This implies the other inequality for the Hausdorff dimension.

The case $\tau = 2$ follows from the fact that $\tau_x \geq 2$ for every irrational number $x$ (see \cite[\S3]{patterson}), while by the above argument the set $\{x : \tau_x > 2\}$ has vanishing Lebesgue measure.
\end{proof}

\begin{proof}[Proof of theorem~\ref{thm:spectrum}]
The set $\{x : \beta(x) = \delta\}$ is completely determined by theorems~\ref{thm:local_rat} and \ref{thm:local_irrat}. Its Hausdorff dimension in the case of cuspidal $f$ is immediate, while if $f$ is noncuspidal it follows from corollary~\ref{cor:spec_hausdorff}.
\end{proof}

\section{Riemann's example}\label{sec:riemann}

In this section we employ the developed machinery to explain some aspects of the graph of Riemann's example (\ref{eq:riemann}), plotted in figure~\ref{fig:riemann}. The material in this section is not new: a similar but more detailed exposition is given by Duistermaat in \cite{duistermaat}. Our analysis, however, is readily applicable to any other modular form.

Riemann's example $\varphi$ satisfies $2\varphi(x) = \Im \theta_1(x)$, where $\theta$ stands for Jacobi's theta function $\theta(z) = \sum_{n \in \mathbb{Z}} e^{n^2 \pi i z}$. This is a modular form of weight $1/2$ for the group $\Gamma_\theta$, consisting of all matrices in $\slgroup_2(\mathbb{Z})$ of the form $\left(\begin{smallmatrix} \text{odd} & \text{even} \\ \text{even} & \text{odd} \end{smallmatrix}\right)$ or $\left(\begin{smallmatrix} \text{even} & \text{odd} \\ \text{odd} & \text{even} \end{smallmatrix}\right)$. The $\Gamma_\theta$-orbit of $0$ corresponds to $\infty$ together will all the rationals $p/q$ with either $p$ even and $q$ odd, or $p$ odd and $q$ even. All the remaining rationals ($p/q$ with both $p$ and $q$ odd) constitute the $\Gamma_\theta$-orbit of $1$. The modular form $\theta$ is cuspidal at $1$ but not at $0$ and the associated multipliers $\mu_\gamma$ are always 8th roots of unity. For the proofs of these facts we refer the reader to \cite{duistermaat}.

Note that we may apply the regularity theorems to recover Hardy's and Gerver's theorems and determine the Hölder exponents of $\varphi$ at every point. Its spectrum of singularities, first obtained by Jaffard in \cite{jaffard}, also follows from theorem~\ref{thm:spectrum}.

Jacobi's function $\theta$ is classically denoted $\vartheta_3$, as it has two companions which are also modular forms of weight $1/2$ for conjugated groups of $\Gamma_\theta$:
\begin{equation*}
\tilde{\theta}(z) = \vartheta_2(z) = \sum_{n \in \mathbb{Z}} e^{\left(n+\frac{1}{2}\right)^2 \pi i z} \quad \text{and} \quad \theta(z+1) = \vartheta_4(z) = \sum_{n \in \mathbb{Z}} (-1)^n e^{n^2 \pi i z}.
\end{equation*}
The nomenclature $\tilde{\theta}$ is not standard but we employ it here as a convenient way to avoid problems with subscripts.

\begin{figure}
\centering
\includegraphics[height=0.32\textheight]{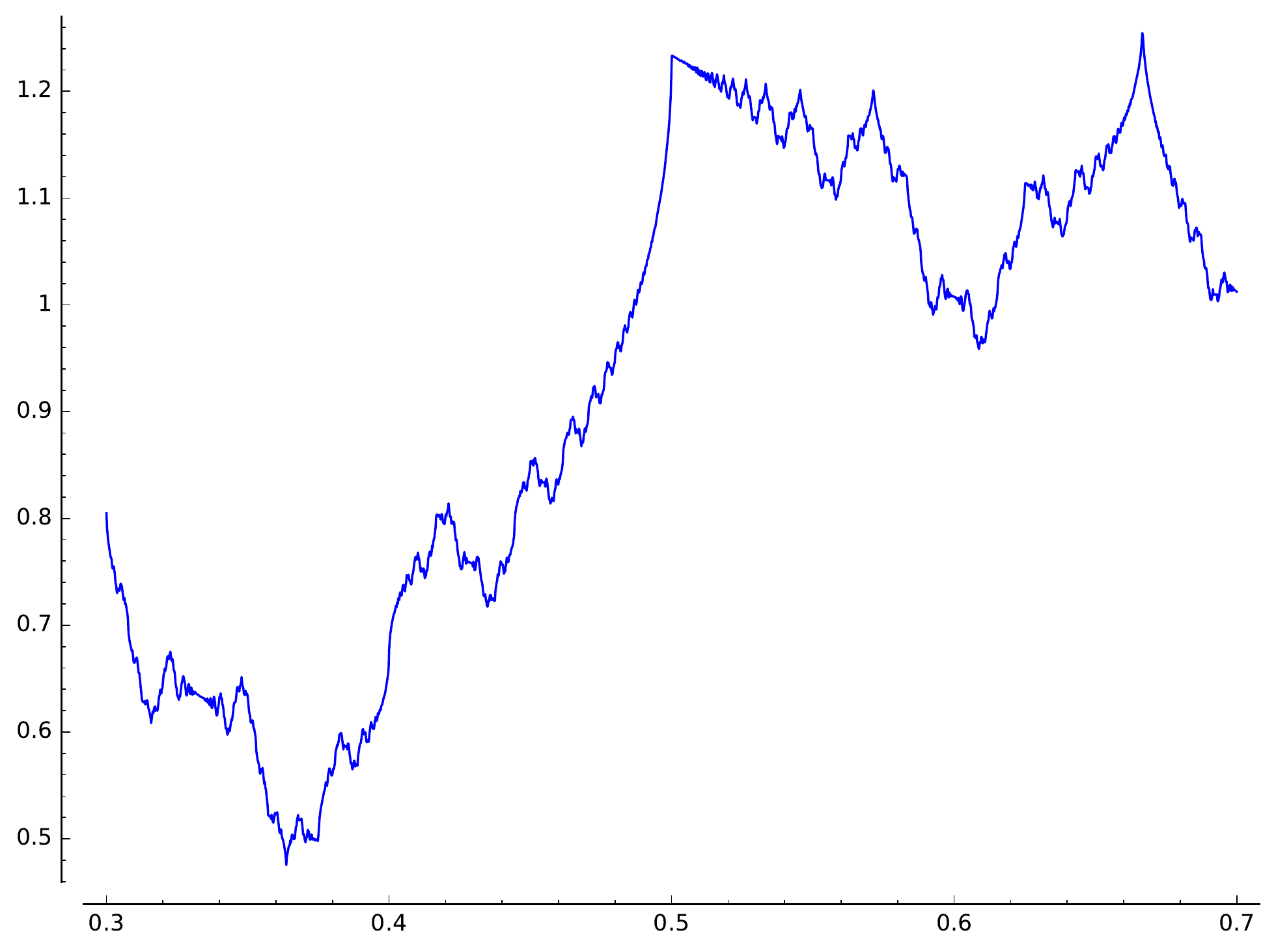}
\includegraphics[height=0.32\textheight]{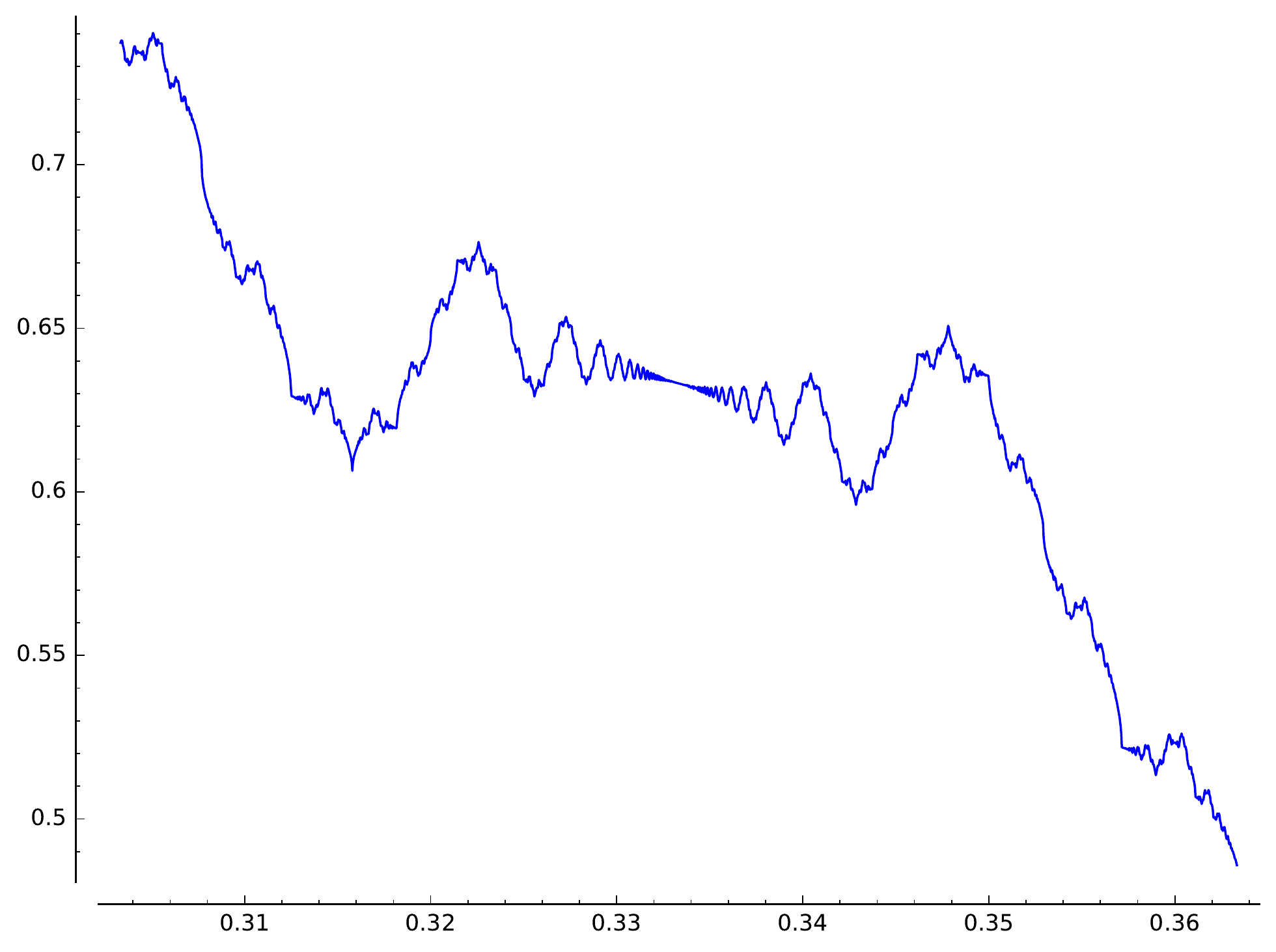}
\includegraphics[height=0.32\textheight]{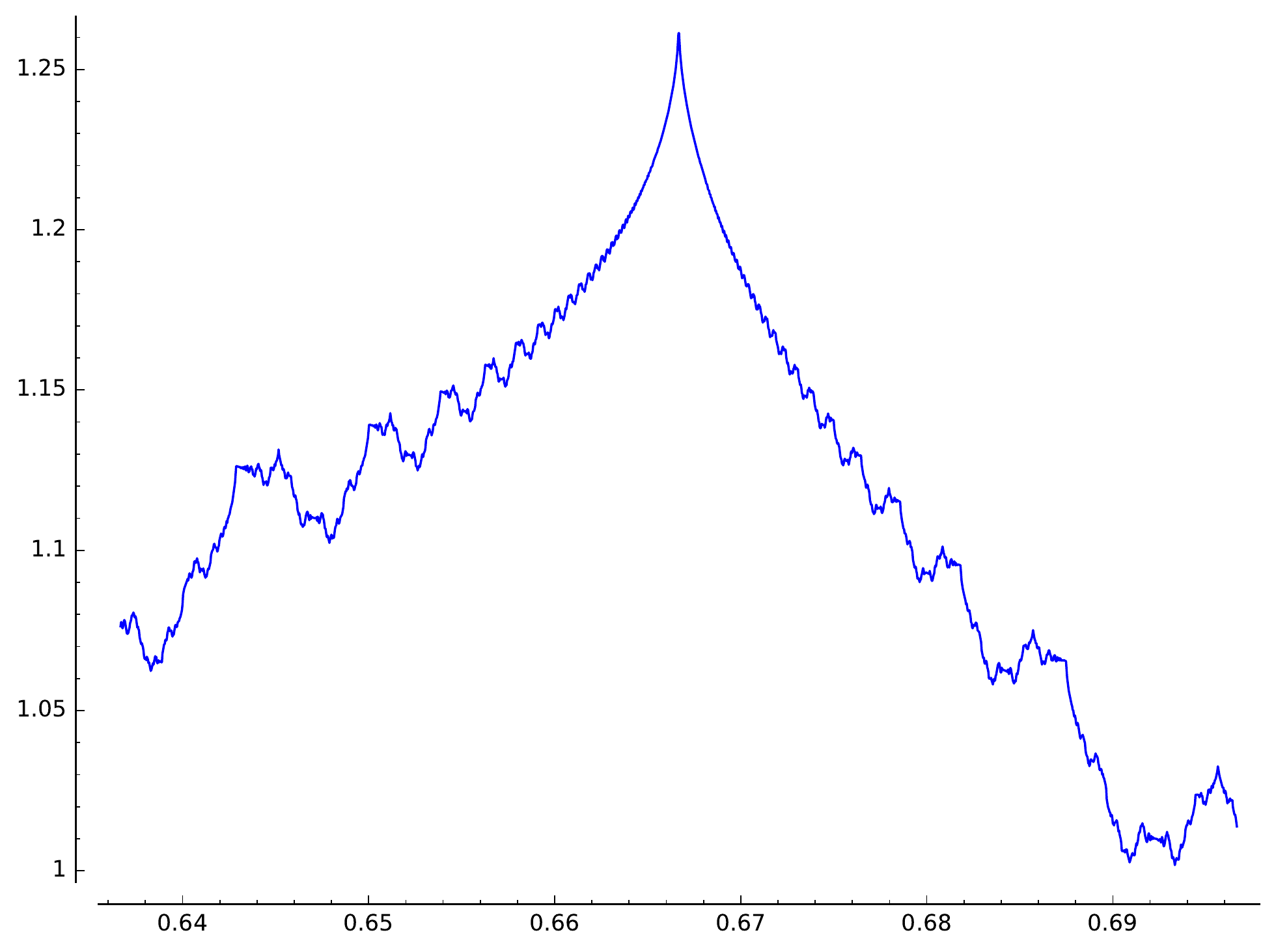}
\caption{Detail of $\varphi$ near $1/2$, $1/3$ and $2/3$, respectively.}
\label{fig:details}
\end{figure}

\begin{figure}[h]
\centering
\includegraphics[width=0.45\textwidth]{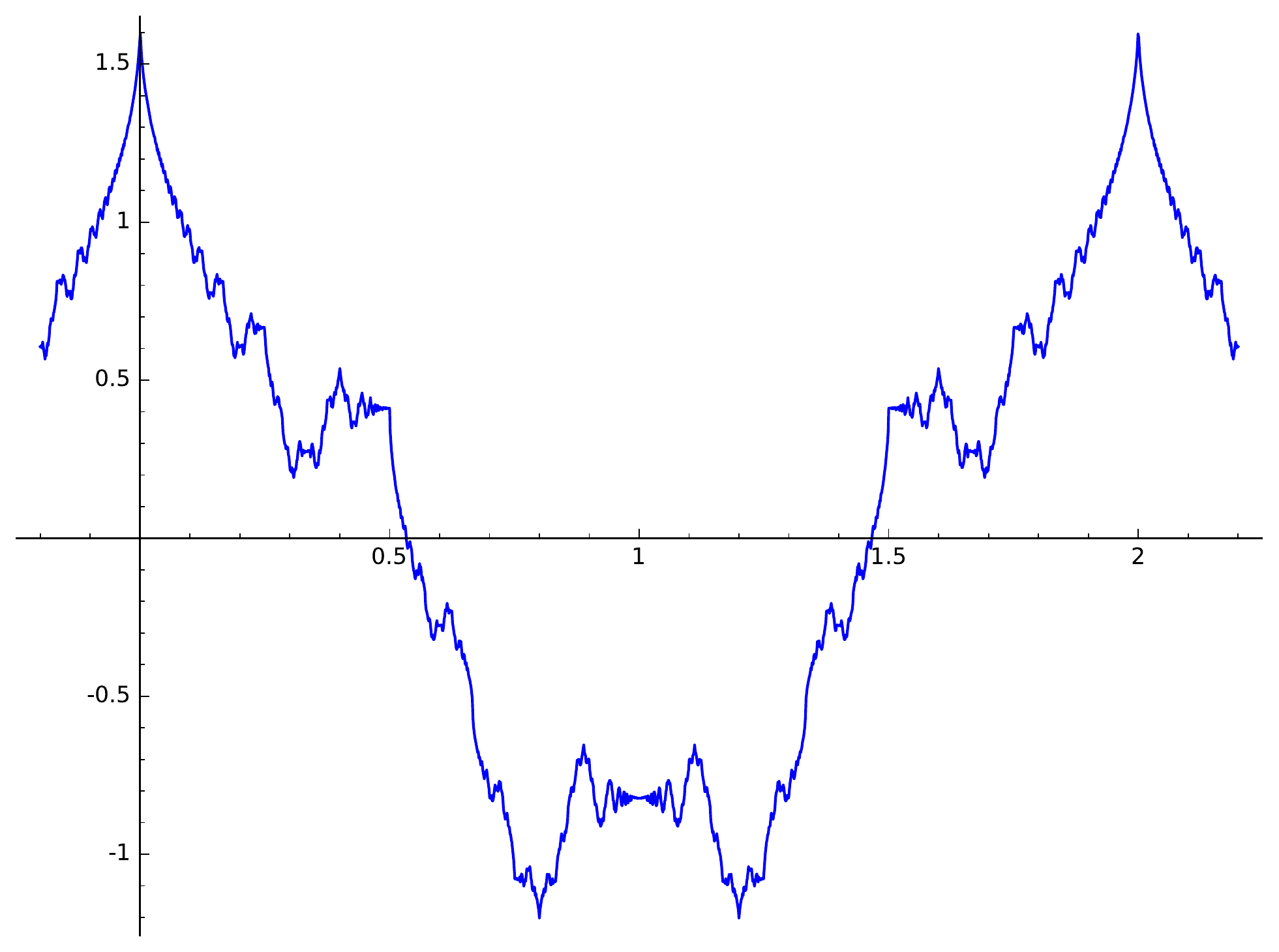}
\includegraphics[width=0.45\textwidth]{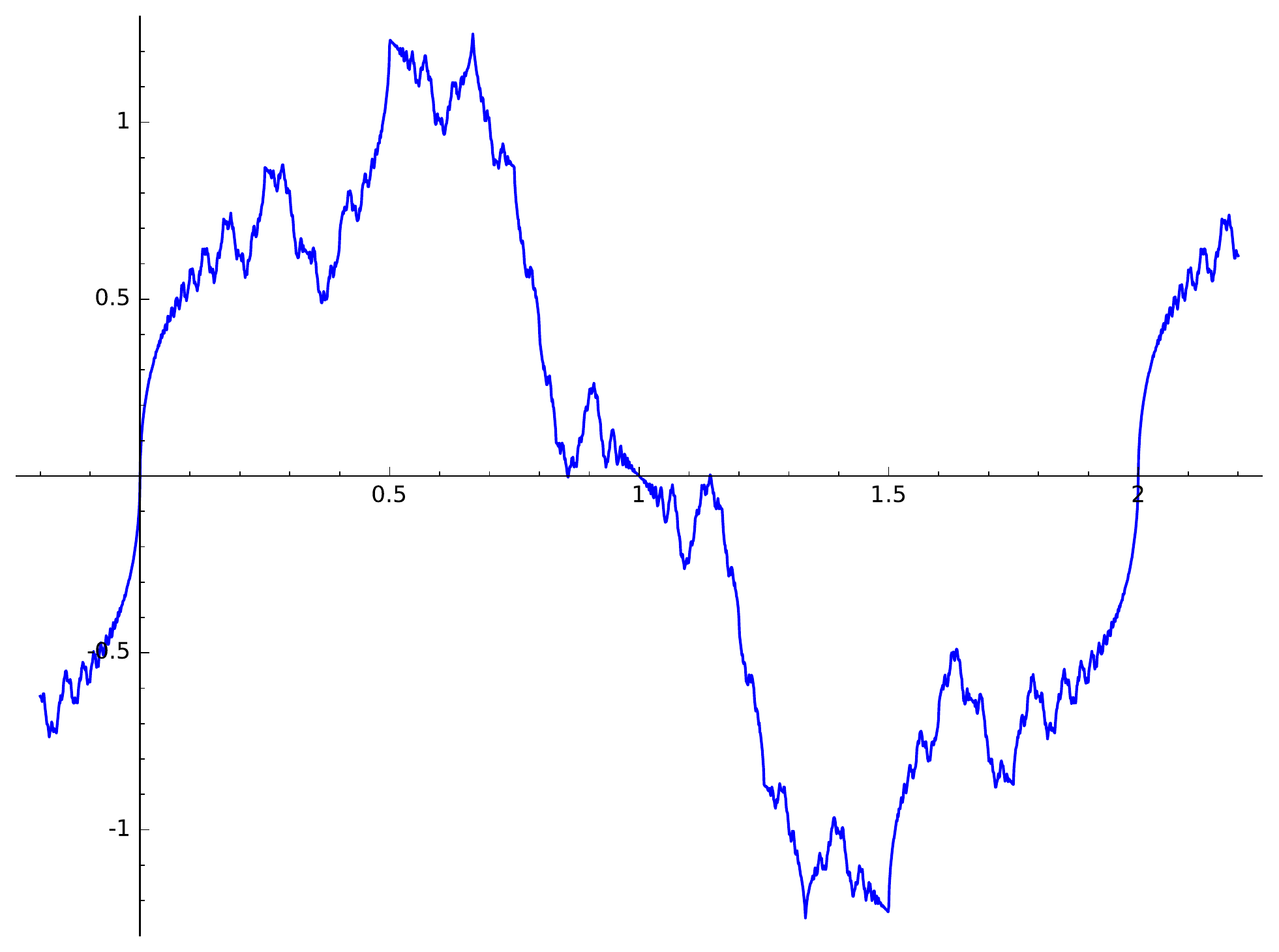}
\includegraphics[width=0.45\textwidth]{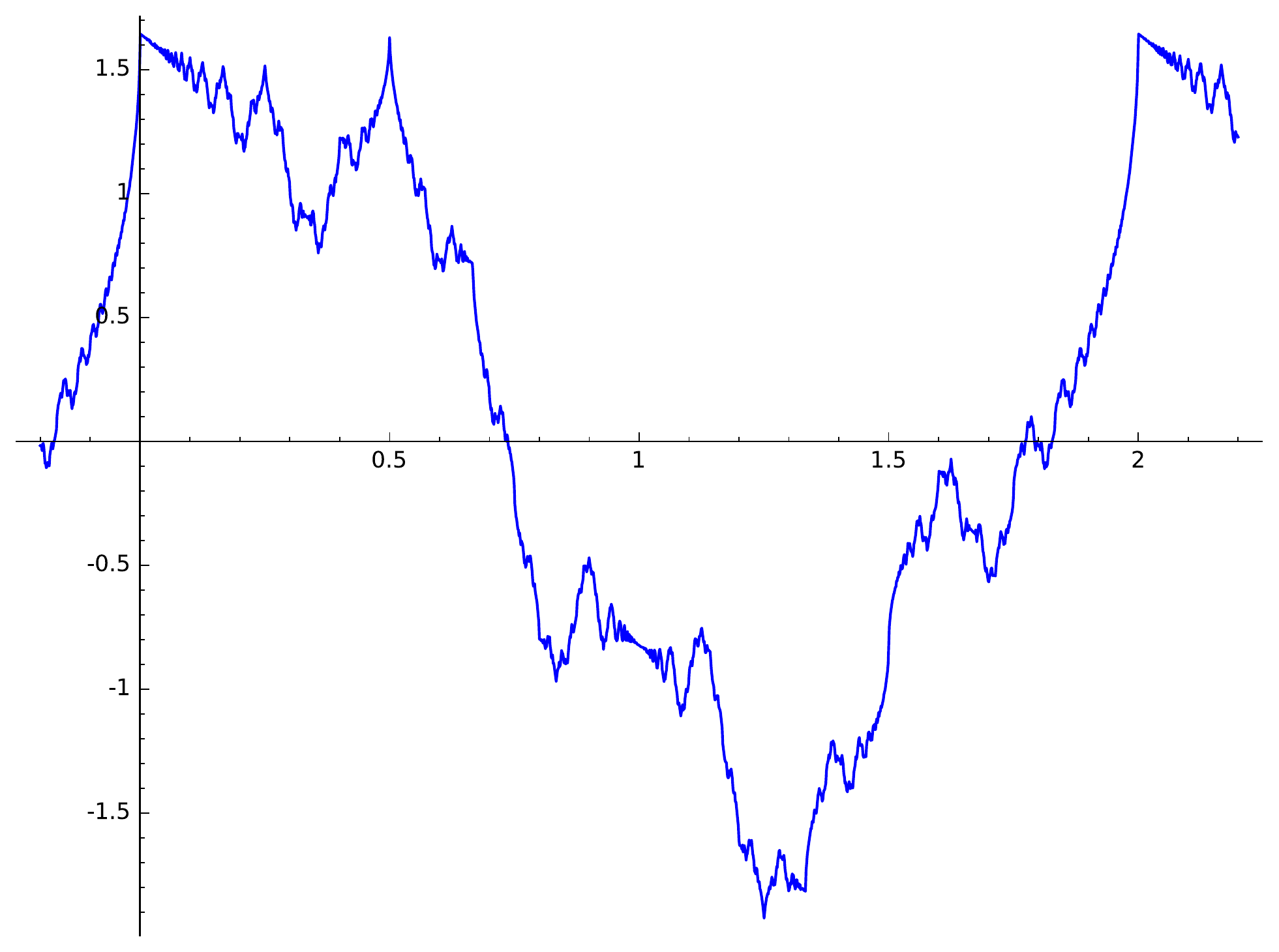}
\includegraphics[width=0.45\textwidth]{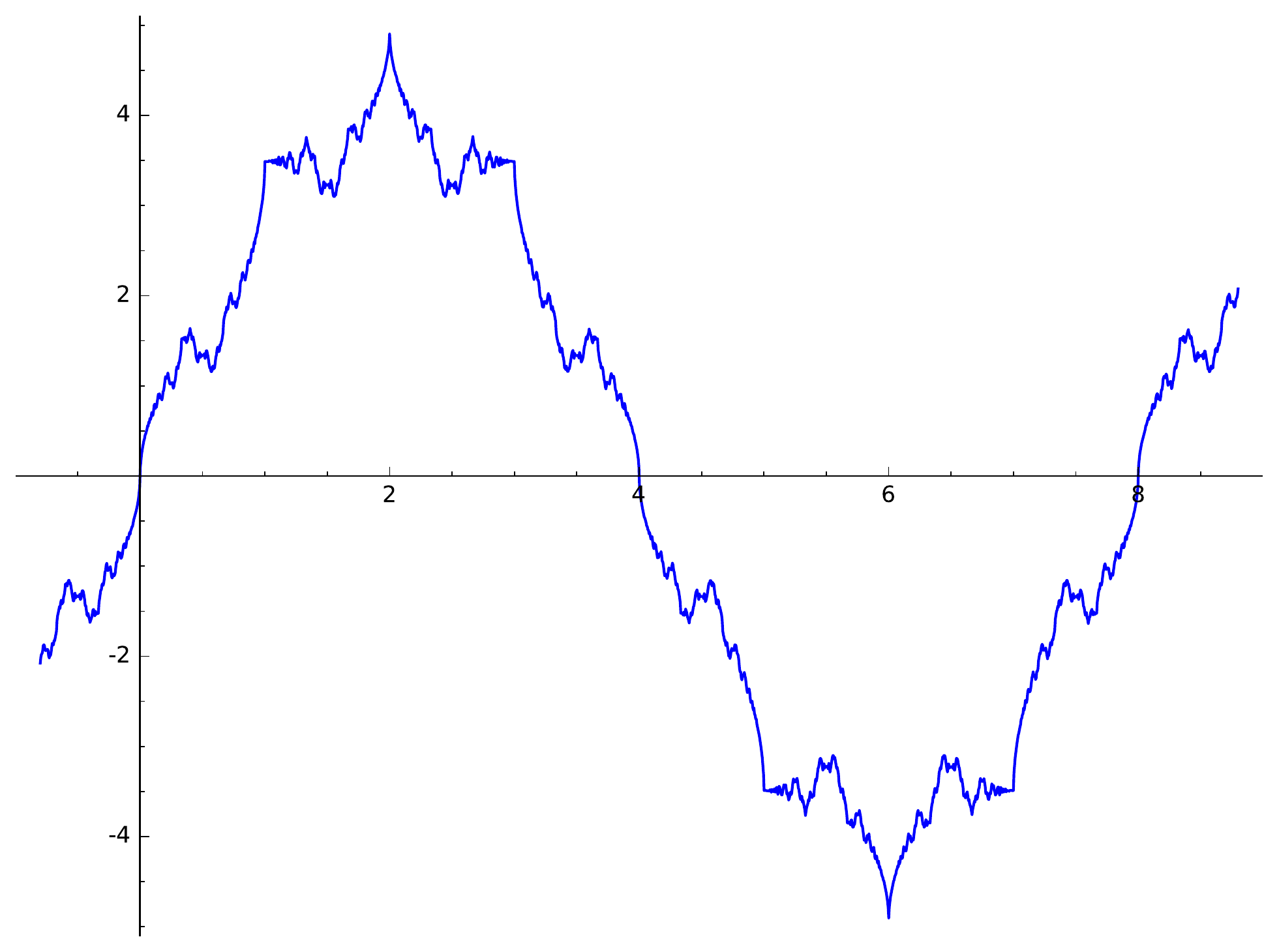}
\caption{Graphs of $\Re \theta_1$ (top-left), $\Im \theta_1$ (top-right), $\Re \theta_1 + \Im \theta_1$ (bottom-left) and $\Im \tilde{\theta}_1$ (bottom-right).}
\label{fig:patterns}
\end{figure}

Given any matrix $\sigma \in \slgroup_2(\mathbb{Z})$ the modular form $\theta^\sigma$ is either a multiple of $\vartheta_2 = \tilde{\theta}$, $\vartheta_3 = \theta$ or $\vartheta_4(z) = \theta(z+1)$, the constant being an 8th root of unity (see theorem~7.1.2 of \cite{rankin}). Since $\theta^\sigma$ is cuspidal at $\infty$ if and only if $\theta\big(\sigma(\infty)\big) = 0$, one concludes that:
\[\theta^\sigma(z) = \begin{cases}
C\theta(z) \text{ or } C\theta(z + 1) & \text{if } \sigma(\infty) \sim 0 \\
C\tilde{\theta}(z) & \text{if } \sigma(\infty) \sim 1.
\end{cases}\]

We now apply theorem~\ref{thm:func_eq} with $\alpha = 1$, $r = 1/2$, to study the behavior of $\varphi = \frac{1}{2} \Im \theta_1$ in the neighborhood of a given rational point $x_0$. The resulting expansion around $x_0$ is of the form:
\[\varphi(x) = \Im\bigg[C \sqrt{x-x_0}\bigg] + \Im\bigg[C' (x-x_0)^{3/2} f_1(\sigma^{-1}x + \tau)\bigg] + h(x).\]
The constant $C$ is nonzero if and only if $x_0 \sim 0$, and in this case $f = \theta$. Otherwise $f = \tilde{\theta}$. The constant $C'$ is always nonzero, and both constants have the argument of an 8th root of unity. Finally, $\tau$ is either $0$ or $1$.

Some deductions are immediate. The first one being that $\varphi$ has singularities of square root type at every rational of the form $\text{odd}/\text{even}$ or $\text{even}/\text{odd}$ (either at one side or both sides of the rational). The second one is that at either side of any rational number $\varphi$ mimics the graph of some periodic function $\Im C' f_1$. Note that as $\sigma^{-1}$ has a simple pole at $x_0$, this pattern repeats indefinitely towards the rational, with its amplitude decreasing as a $3/2$ power of the remaining distance and its frequency roughly proportional to $|x-x_0|^{-1}$. See figure~\ref{fig:details} for some examples of this behavior, where some square root singularities are also clearly visible.

Since the argument of $C'$ is an integer multiple of $\pi/4$ we also deduce that $\Im C' f_1$ is either $\Re{f_1}$, $\Im{f_1}$ or $\Re{f_1} + \Im{f_1}$, or the mirror image of one of these three functions, \emph{i.e.}, the result of performing the change of variables $x \mapsto -x$ either in the domain, in the codomain or both. The situation is even simpler when $f = \tilde{\theta}$, as all these functions are then translates and mirror images of each other (\emph{cf.} theorem~7.1.2 of \cite{rankin}) and therefore we need only to consider $\Im \tilde{\theta}_1$. Hence the graph of $\Im C' f_1$ corresponds, up to symmetries, to one of the four genuinely distinct patterns that appear in figure~\ref{fig:patterns}. Note that in figure~\ref{fig:details} all four patterns appear.

A different kind of self-similarities, modulo a $\mathcal{C}^{1, 0}$ function which need not have any decay, may be found around fixed points of transformations lying in $\Gamma_\theta$, as deduced from theorem~\ref{thm:func_eq} by letting $x$ approach the fixed point of the transformation. Note that for any finite index subgroup $\Gamma$ of $\slgroup_2(\mathbb{Z})$ the set of real points fixed by transformations in $\Gamma$ comprises $\mathbb{Q}$ and all the quadratic surds (fixed respectively by parabolic and hyperbolic transformations). Indeed, if $\sigma \in \slgroup_2(\mathbb{Z})$ is parabolic (hyperbolic) and fixes a real point, some power $\sigma^n$ lies in $\Gamma$, is also parabolic (hyperbolic) and fixes the same point.

\section{Cusp forms for $\Gamma_0(N)$}\label{sec:gamma0}

Fix an arbitrary integer $N \geq 1$ and let $f$ be a cusp form of integer weight $r$ for the group $\Gamma_0(N)$ and trivial multiplier system. Note that $r$ must necessarily be even. For any $\alpha > r/2$ the function $f_\alpha$ is well-defined and we may consider $g = \Re f_\alpha$ or $\Im f_\alpha$. Under these conditions the constant $B$ in theorem~\ref{thm:func_eq} is always positive, and hence for every rational $x_0 \sim \infty$ modulo $\Gamma_0(N)$ we have
\begin{equation}\label{eq:cusp_pattern}
g(x) = B|x-x_0|^{2\alpha - r}g(\sigma^{-1}x) + E(x)
\end{equation}
for some $\sigma \in \slgroup_2(\mathbb{R})$ satisfying $\sigma^{-1} x_0 = \infty$ and the function $E$ lying in the spaces specified by theorem~\ref{thm:func_eq}. An interesting question is whether an approximate functional equation of the form \eqref{eq:cusp_pattern}, with $B$ real and $E$ with the same regularity, relating $g$ with itself, exists for other rational numbers. Note this will happen for the rational $x_0$ as long as we are able to find some $\sigma \in \slgroup_2(\mathbb{R})$ satisfying $\sigma^{-1} x_0 = \infty$ and such that $f^\sigma = f|_\sigma$ equals $Cf$ for a real constant $C$ (and this is likely a necessary condition). In this section we provide sufficient conditions for this to hold and study some examples.

\newsavebox{\boxsigma} 
\savebox{\boxsigma}{$\sigma = \left(\begin{smallmatrix} 7 & 3 \\ 14 & 7 \end{smallmatrix}\right)$}

\begin{figure}[h]
\centering
\includegraphics[width=0.45\textwidth]{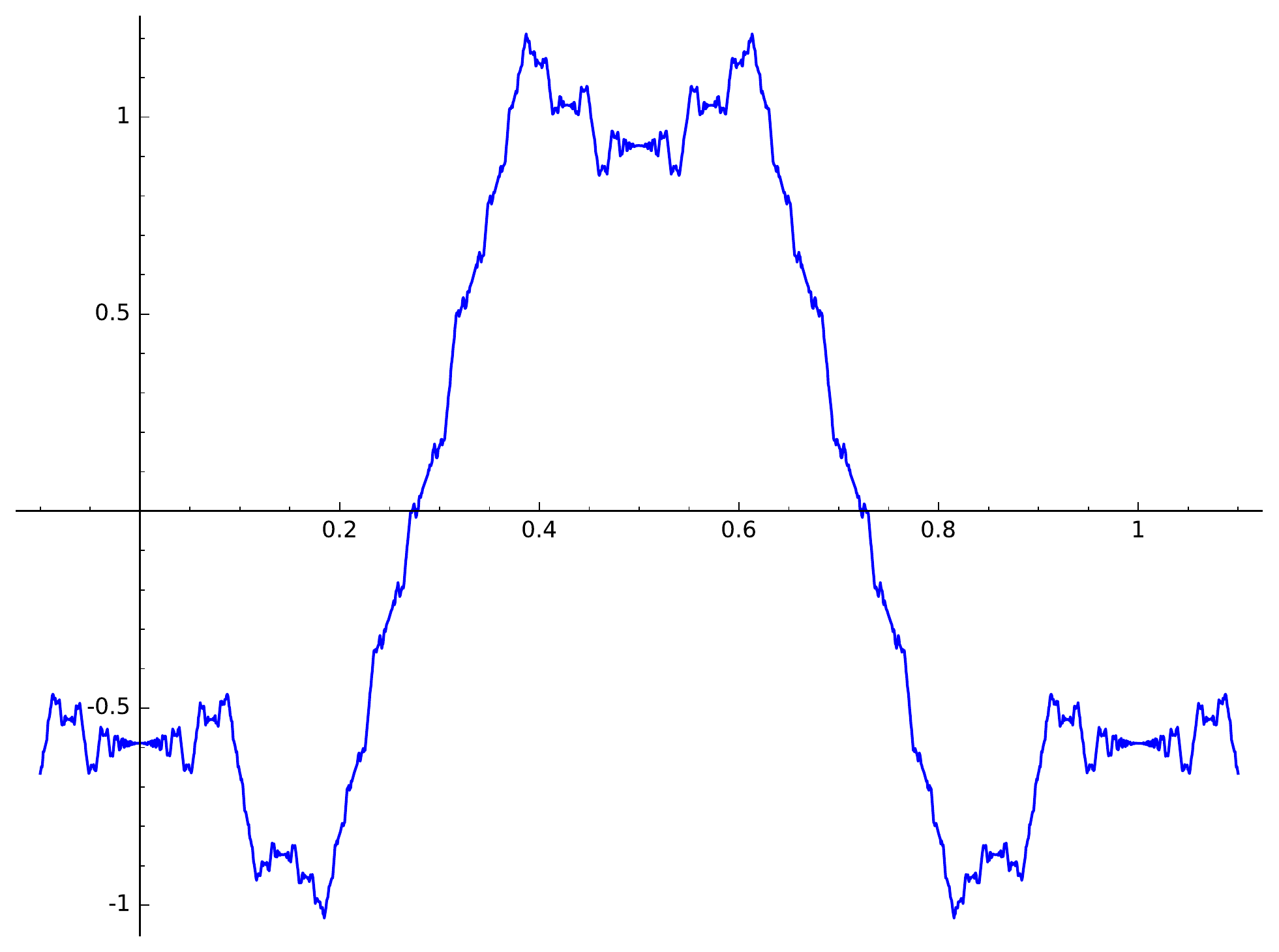}
\includegraphics[width=0.45\textwidth]{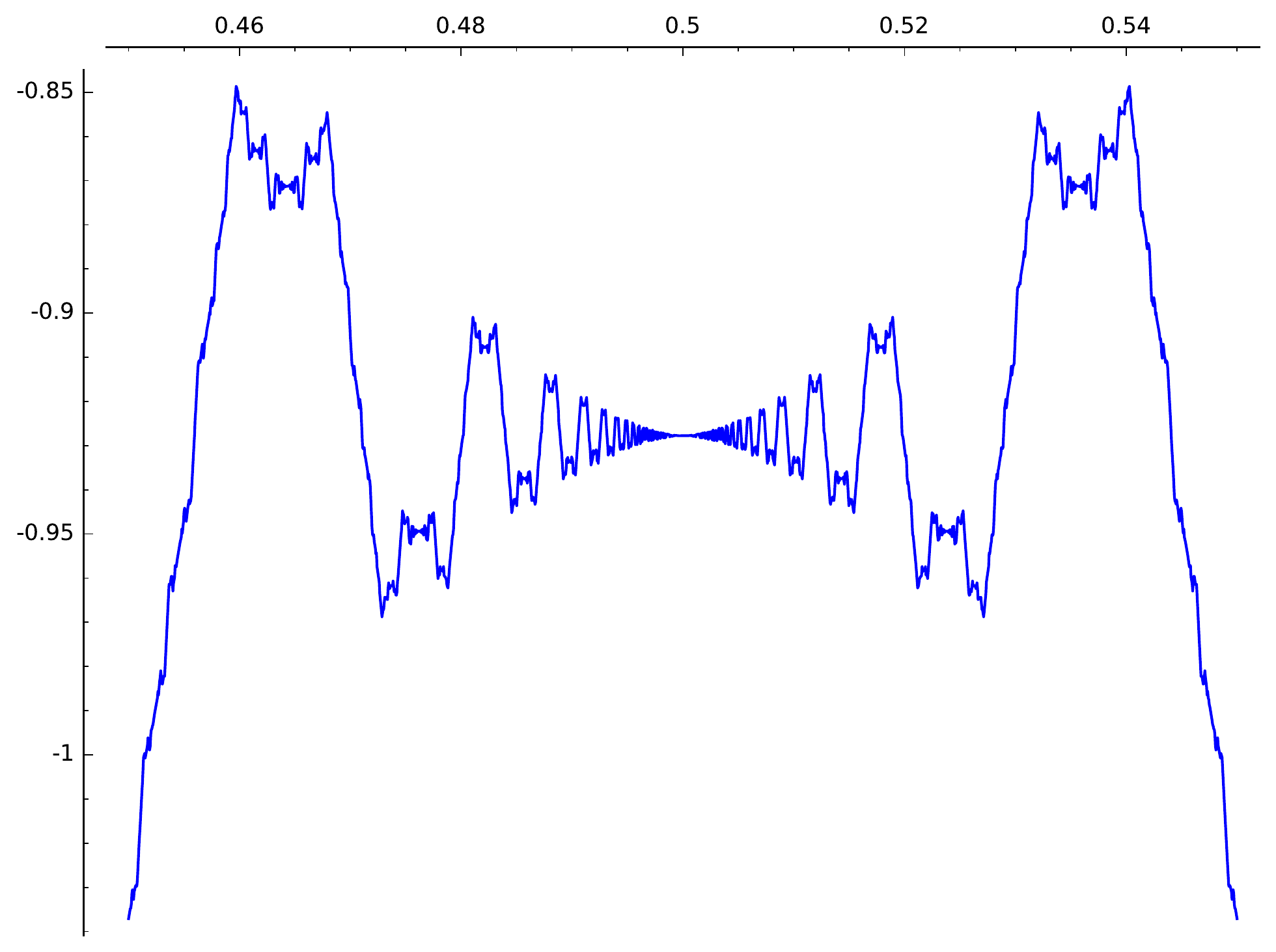}
\caption{Left: Plot of $-\Re f_{9/5}$, where $f$ is the newform on $\Gamma_0(14)$. Right: detail of $\Re f_{9/5}$ at $1/2$. This rational is not in $\Gamma_0(14) \cdot \infty$, but the matrix \usebox{\boxsigma} satisfies $\sigma(\infty) = 1/2$ and $f|_\sigma = -f$.}
\label{fig:e14}
\end{figure}

Some notation first. For any two integers $n$ and $m$ we denote by $(n, m)$ its greatest common divisor, and for every prime $p$ we denote by $[n]_p$ the largest power of $p$ dividing $n$. For every divisor $Q \mid N$ satisfying $(Q, N/Q) = 1$ we define the matrix
\[\omega_Q := \left(\begin{matrix}
Qx & y \\
Nz & Qw
\end{matrix}\right),
\quad \quad x, y, z, w \in \mathbb{Z}, \quad \det \omega_Q = Q,\]
which is unique up to left and right multiplication by elements of $\Gamma_0(N)$. The matrices $\omega_Q$ are called Atkin-Lehner involutions and satisfy $Q^{-1} \omega_Q^2 \in \Gamma_0(N)$ and $\omega_Q \omega_{Q'} =$ some $\omega_{QQ'}$ whenever  $(Q, Q') = 1$. For the sake of clarity we also set $\omega_p := \omega_{[N]_p}$ for each prime $p \mid N$. Finally for any integer $n > 0$ we consider the matrix
\[S_n := \left(\begin{matrix}
1 & 1/n \\
0 & 1
\end{matrix}\right),\]
which corresponds to a translation by $1/n$.

A theorem of Atkin and Lehner stated without proof in \cite{atkin_lehner} assures that when $N$ is not divisible by $4$ nor $9$ the normalizer of $\Gamma_0(N)$ is generated by $\Gamma_0(N)$ and the Atkin-Lehner involutions $\omega_p$ for primes $p \mid N$. When $N$ is divisible by $4$ or by $9$ one has to include some extra generators: $S_2$ if $[N]_2 = 4$ or $8$, $S_4$ if $[N]_2 = 16$ or $32$ and $S_8$ if $64 \mid N$; and $S_3$ if $9 \mid N$. Note that we are considering the normalizer of $\Gamma_0(N)$ as a group of linear fractional transformations, as otherwise one also needs to include any real multiple of the previous generators. This theorem also provides the structure of the quotient group between the normalizer of $\Gamma_0(N)$ and $\Gamma_0(N)$ itself (which we do not need), although this part seems to have some mistakes and a corrected version is proved by Bars in \cite{bars}.

Asai observed in \cite{asai} that the Atkin-Lehner involutions act transitively on $\mathbb{Q}$ if and only if $N$ is square-free. The following proposition is a generalization of this fact.

\begin{proposition}\label{prop:normalizer}
The normalizer of $\Gamma_0(N)$ acts transitively on $\mathbb{Q}$ if and only if $N = 2^a 3^b N'$ for some $a < 8$, $b < 4$ and a square-free integer $N'$ not divisible by $2$ nor $3$.
\end{proposition}

\begin{proof}
Assume first that $N$ is of the prescribed form and let $u/v$ be an arbitrary rational number, $(u, v) = 1$. It suffices to show that $u/v$ is related modulo the normalizer to some $u'/v'$ with $(u', v')= 1$ and $N \mid v'$, as these rationals comprise the orbit of $\infty$ modulo $\Gamma_0(N)$. We do this by stages, first relating it to a rational whose denominator is divisible by $N'$, then adding $2^a$ and finally $3^b$.

Write $N' = p_1 \cdots p_n$ for distinct primes $p_1, \ldots, p_n$. We may assume upon reordering of the $p_i$ that $p_1 \cdots p_m \mid v$ and $p_i \nmid v$ for $m < i \leq n$. Choosing $Q = 2^a 3^b p_{m+1} \cdots p_n$ we have
\[u'/v' = \omega_Q(u/v) = \frac{Q xu + yv}{N\left(zu + w\frac{v}{N/Q}\right)}.\]
The numerator of the right hand side is not divisible by any of the $p_i$ as a consequence of the determinant condition imposed on $\omega_Q$ and therefore $N' \mid v'$.

Hence assume that from the beginning $N' \mid v$. This divisibility property is preserved by $\omega_2$, $S_2$, $S_4$ and $S_8$. We show now we may find a related $u'/v'$ with $2^a N' \mid v'$. Let $2^s = [v]_2$ and assume that $s < a$, since otherwise we are finished. It is easy to check that if $u'/v' = \omega_2(u/v)$ then $[v']_2 = 2^{a-s}$. This means that applying $\omega_2$ if necessary we may assume $s \leq [a/2]$. We now apply repeatedly $S_2$, $S_4$ or $S_8$ to arrive to a rational with $s = 0$, and the image of this rational by $\omega_2$ satisfies $s \geq a$.

The same argument can now be applied \emph{mutatis mutandis} to add the factor $3^b$ to the denominator. This finishes the proof of the direct implication.

To prove that the normalizer action is not transitive when $N$ is not of the prescribed form it suffices to show a proper subset of $\mathbb{Q}$ invariant under this action. Suppose first that for some prime $p \neq 2, 3$ we have $p^2 \mid N$ and $p^c = [N]_p$. Then one such set is that of the rational numbers $u/v$ with $[v]_p = p^s$ and $0 < s < c$. The invariance of this set follows from the following facts: the translations and the Atkin-Lehner involutions $\omega_Q$ with $p \nmid Q$ leave $[v]_p$ invariant, while $[v']_p = p^{c-s}$ for $u'/v' = \omega_Q(u/v)$ with $p \mid Q$.

The remaining cases are $2^8 \mid N$ or $3^4 \mid N$. If $2^8 \mid N$ then $a \geq 8$ and one such set is that of the rational numbers $u/v$ with $[v]_2 = 2^{a/2}$ if $a$ is even and $[v]_2 = 2^{[a/2]}$ or $[v]_2 = 2^{[a/2] + 1}$ if $a$ is odd. An analogous set works when $3^4 \mid N$.
\end{proof}

If a cusp form $f$ satisfies $f|_\sigma = C_\sigma f$, where $C_\sigma$ is a real constant, for every $\sigma$ lying in the normalizer of $\Gamma_0(N)$, then we may guarantee the approximate functional equation \eqref{eq:cusp_pattern} to exist around every rational number in the orbit of $\infty$ modulo this normalizer. If the action of the normalizer is also transitive on $\mathbb{Q}$ then the equation exists for every rational number. Suppose now that $f$ is a newform (for the precise definition see \cite[\S9.4]{rankin}, for example). Atkin and Lehner proved in \cite{atkin_lehner} that $f|_{\omega_p} = \pm f$ for every prime $p \mid N$. In the same paper they also prove that when $4 \mid N$ all the even coefficients of $f$ vanish, and therefore $f|_{S_2} = -f$. If these transformations suffice to generate the normalizer, then the previous remarks apply.

When we have to include $S_3$, $S_4$ or $S_8$ to generate the normalizer, however, this breaks down, as it is not generally true that $f|_{S_n} = Cf$ for a real constant $C$. A workaround exists when the space of cuspidal forms has dimension $1$. In this case, $f|_\eta$ is again a constant multiple of $f$ for any $\eta$ in the normalizer of $\Gamma_0(N)$, and therefore all these matrices commute under the action of the slash operator. As a consequence, $f|_\eta = f|_{\omega_Q S} = \pm f|_S$ for some $Q \mid N$ and some traslation $S$. The matrix $\sigma = \eta S^{-1}$ now lies in the normalizer of $\Gamma_0(N)$ and satisfies $\sigma(\infty) = \eta(\infty)$ and $f|_\sigma = \pm f$. Therefore, if the normalizer acts transitively on $\mathbb{Q}$, so does the subgroup consisting of those matrices $\sigma$ for which $f|_\sigma = \pm f$.

We conclude that the following are sufficient conditions to ensure that there is an approximate functional equation \eqref{eq:cusp_pattern} around every rational number: $N = 2^a N'$ with $a < 4$ and $N'$ odd and square-free, or if the space of cusp forms on $\Gamma_0(N)$ has dimension $1$ and $N = 2^a 3^b N'$ with $a < 8$, $b < 4$ and $N'$ square-free and not divisible by $2$ nor $3$.

\begin{figure}
\centering
\includegraphics[width=0.9\textwidth]{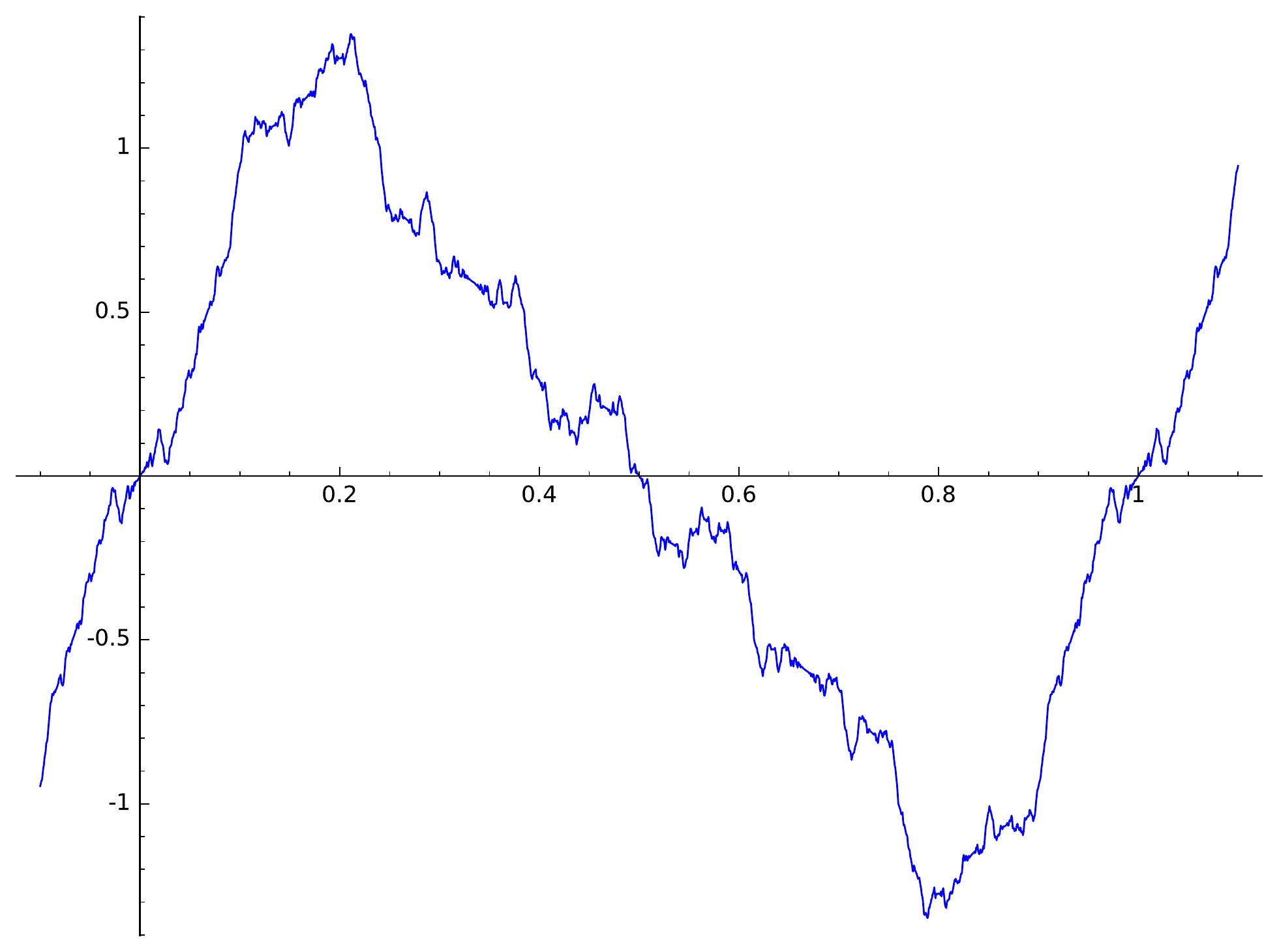}
\caption{Plot of $\Im f_{7/4}$ where $f$ is the newform on $\Gamma_0(45)$.}
\label{fig:newform45}
\end{figure}

\begin{figure}
\centering
\includegraphics[width=0.45\textwidth]{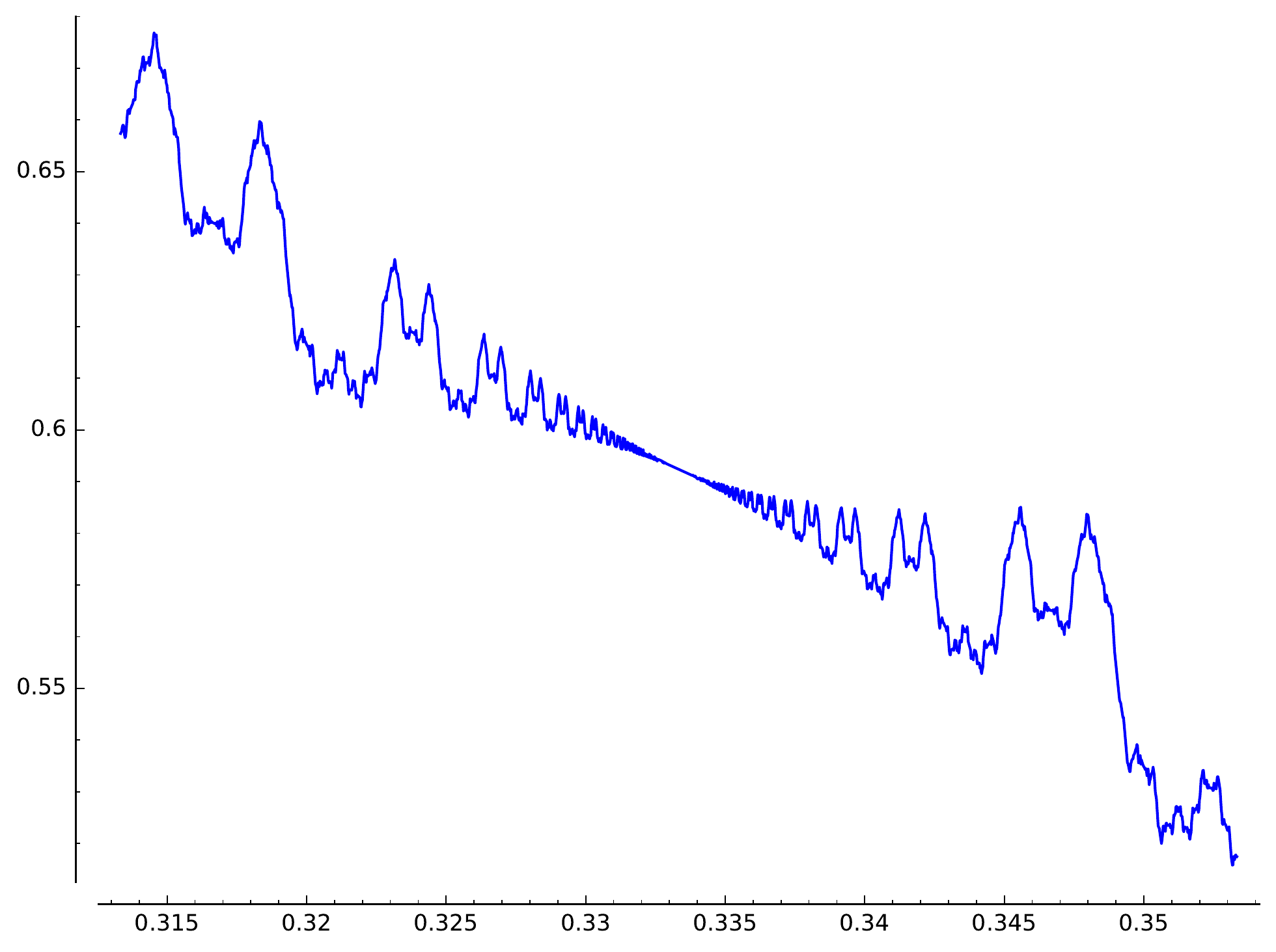}
\includegraphics[width=0.45\textwidth]{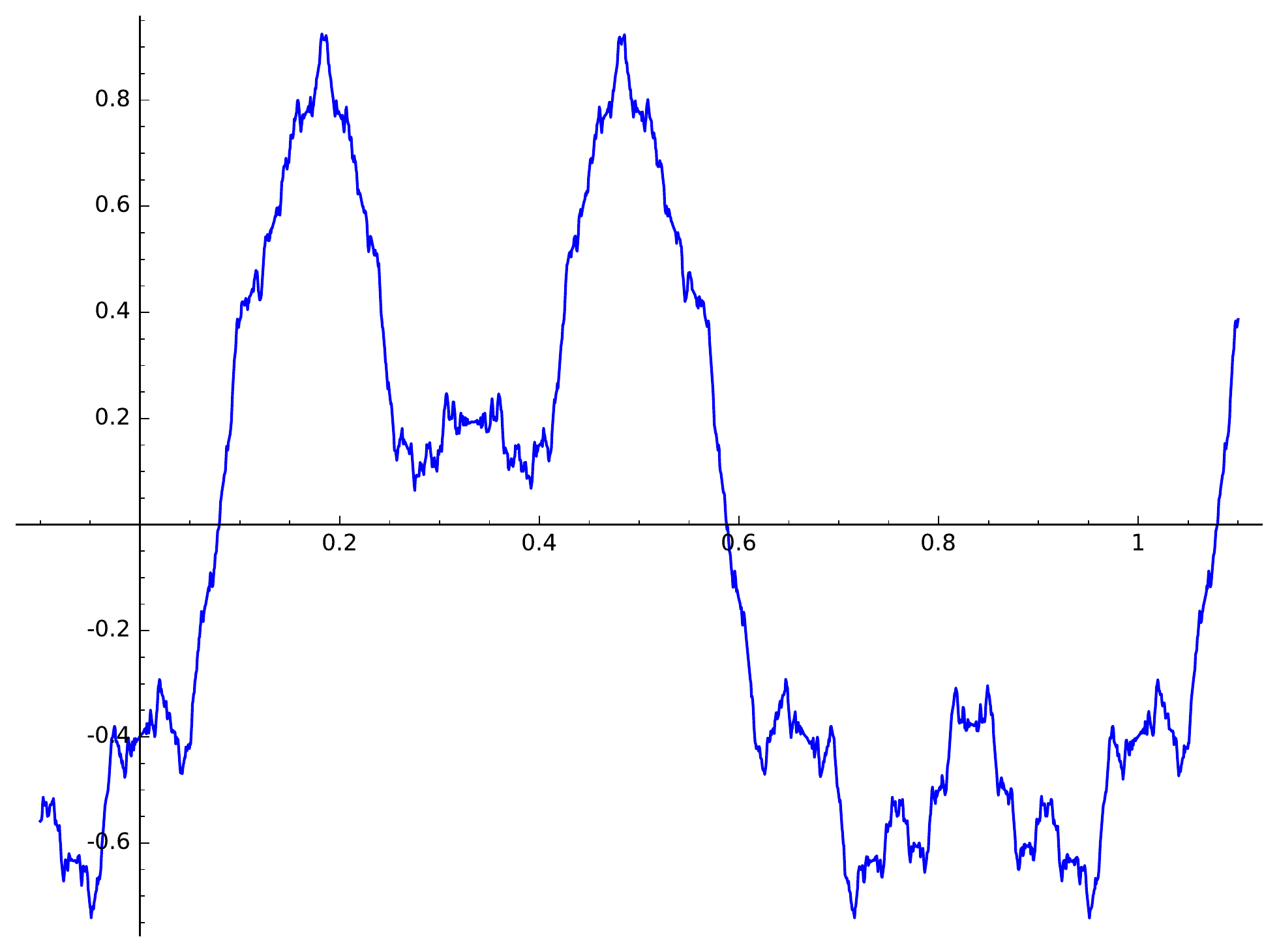}
\caption{Left: Detail of $\Im f_{7/4}$ around $1/3$ where $f$ is the newform on $\Gamma_0(45)$. Right: Graph of the imaginary part of the right hand side of (\ref{eq:elliptic_example}).}
\label{fig:e45_patterns}
\end{figure}

We now give some examples of modular forms for which an equation like \eqref{eq:cusp_pattern}, relating $g$ to itself, is unlikely to exist around some rational numbers. These are of weight $2$ and therefore associated to modular abelian varieties over $\mathbb{Q}$. By direct examination of the table of newforms found at \cite{lmfdb} we see that the lowest value of $N$ for which neither of the previous conditions is satisfied is $N = 45$, as the associated space of cusp forms happens to be of dimension $3$, containing an oldclass generated by the newform on $\Gamma_0(15)$. Denote by $f$ the newform on $\Gamma_0(45)$ and by $h$ the one on $\Gamma_0(15)$. These are associated to the isogeny classes of the elliptic curves
\[y^2 + xy = x^3 - x^2 - 5 \quad \text{ and } \quad y^2 + xy + y = x^3 + x^2,\]
respectively. The matrix $\sigma = S_3 \omega_{45}$, where $\omega_{45}$ is the Atkin-Lehner involution determined by $x = w = 0$, $y = 1$ and $z = -1$, lies in the normalizer of $\Gamma_0(45)$ and sends $\infty$ to $1/3$. The function $f|_\sigma$ is therefore again a modular form for $\Gamma_0(45)$, and in fact it has the following decomposition:
\[f|_\sigma(z) = \frac{1}{2} f(z) - i \frac{1}{2\sqrt{3}} h(z) - i \frac{3 \sqrt{3}}{2} h(3z).\]
To obtain the coefficients one first decomposes $f|_{S_3}$ by directly comparing coefficients, and then applies $|_{\omega_{45}}$. The Atkin-Lehner eigenvalues are tabulated in \cite{lmfdb}, and the action of this operator on oldforms is described by lemma~26 of \cite{atkin_lehner}. As an immediate consequence
\begin{equation}\label{eq:elliptic_example}
f_\alpha^\sigma(x) = \frac{1}{2} f_\alpha(x) - \frac{i}{2\sqrt{3}} h_\alpha(x) - \frac{i}{2 \cdot 3^{\alpha - 3/2}} h_\alpha(3x).
\end{equation}
In figure~\ref{fig:newform45} we have plotted $g = \Im f_{7/4}$, while in figure~\ref{fig:e45_patterns} the reader can compare the imaginary part of the right hand side of (\ref{eq:elliptic_example}) for $\alpha = 7/4$ with aspect of the graph of $g$ near $\sigma(\infty) = 1/3$.

The lowest value of $N$ for which the normalizer is not transitive on $\mathbb{Q}$ and there is some nonzero newform is $N = 49$. This newform is associated to the isogeny class of the curve
\[y^2 + xy = x^3 - x^2 - 2x - 1.\]
The cusp $1/7$ is not related to $\infty$, not even by the normalizer, and in figure~\ref{fig:newform49} the reader can appreciate how for $g = \Re f_{7/4}$ the aspect of the repeating pattern around $1/7$ and that of the global graph seem to differ, making it unlikely for a self-similarity relation like \eqref{eq:cusp_pattern} to hold.

To finish this section we note that although Jacobi's theta function $\theta$ is usually presented as the modular form $\theta(2z)$ for the group $\Gamma_0(4)$, the previous analysis would not apply because it is of half-integer weight and the multiplier system is not trivial.

\begin{figure}
\centering
\includegraphics[width=0.45\textwidth]{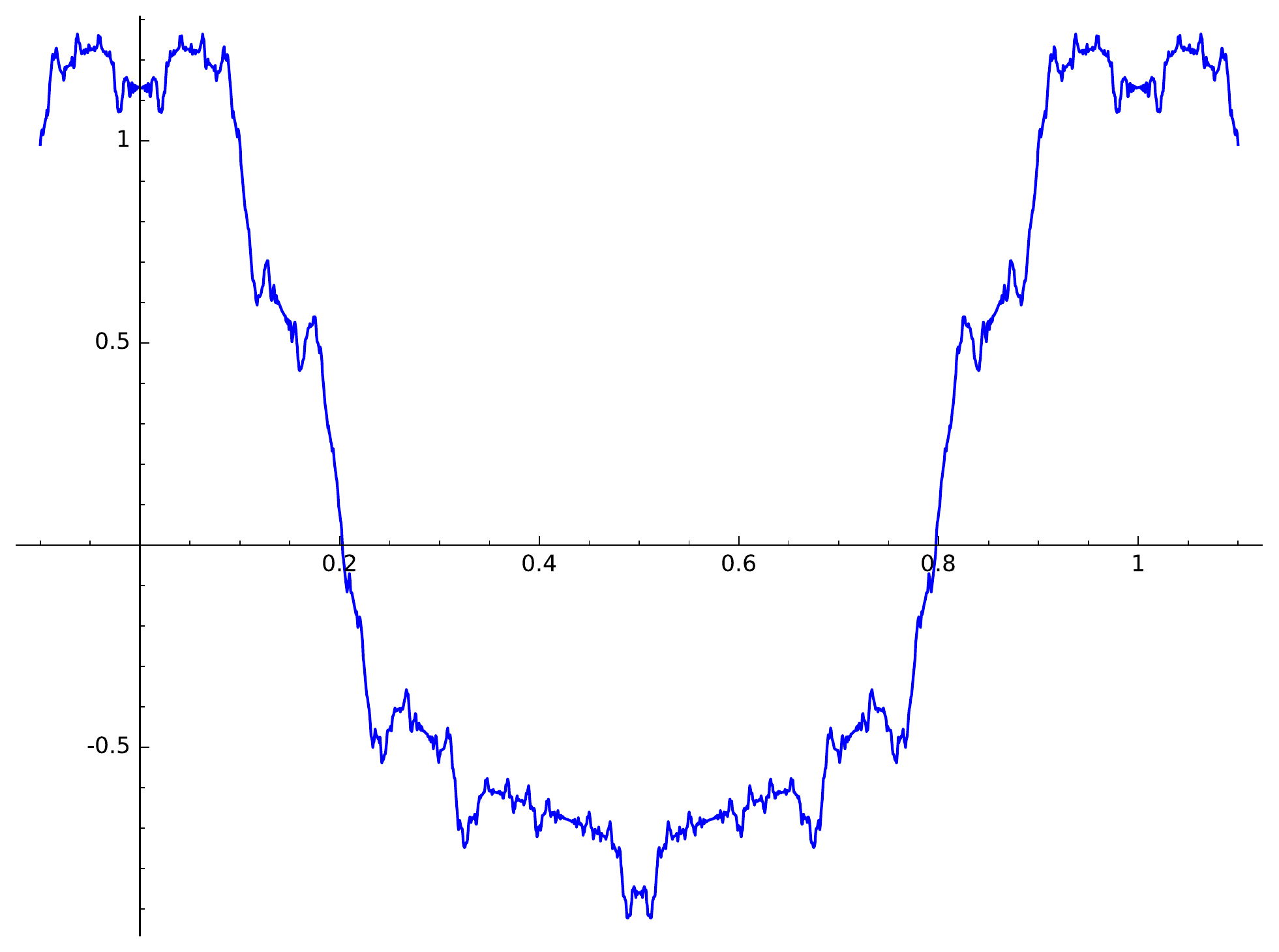}
\includegraphics[width=0.45\textwidth]{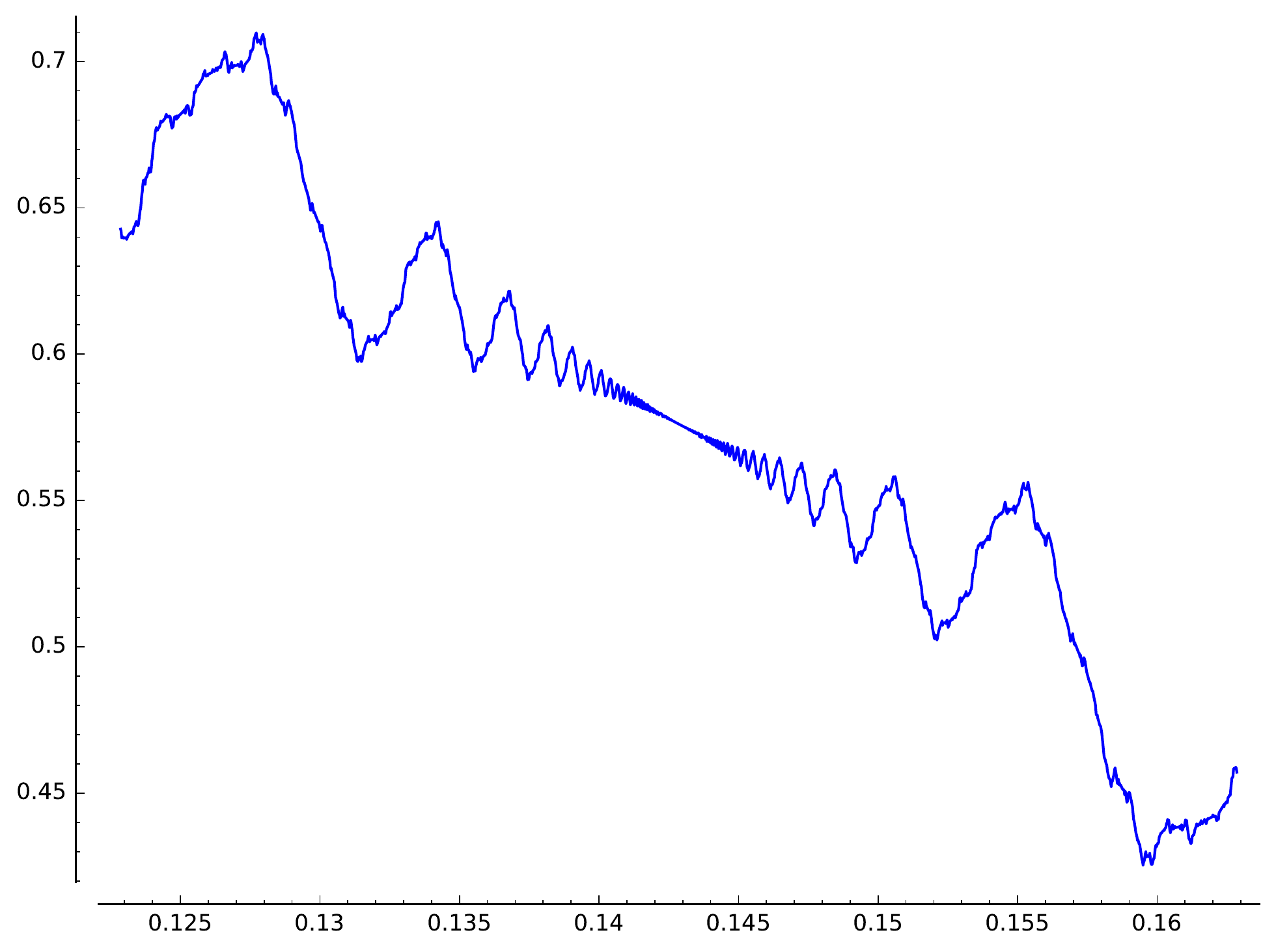}
\caption{Left: Plot of $\Re f_{7/4}$ where $f$ is the newform on $\Gamma_0(49)$. Right: Detail around $1/7$.}
\label{fig:newform49}
\end{figure}

\section*{Acknowledgments}

The author is indebted to Fernando Chamizo for his encouragement and suggestions during the elaboration of this article. This work has been supported by the ''la Caixa''-Severo Ochoa international PhD programme at the Instituto de Ciencias Matemáticas (CSIC-UAM-UC3M-UCM).

The graphics included in the article have been plotted using \emph{SageMath} \cite{sage}, and the same software system has been used to compute the Fourier coefficients of newforms. The partial sums were calculated using simple {\CC} programs.

\end{document}